\documentclass[a4paper,11pt]{amsart}
\usepackage{amssymb,amscd,amsxtra,graphicx,psfrag,xypic}

\usepackage[pdftex]{hyperref}
\hypersetup{%
bookmarksnumbered=true,%
colorlinks=true,%
setpagesize=false,%
pdftitle={},%
pdfauthor={Takuzo Okada}}

\theoremstyle{plain}
\newtheorem{Thm}{Theorem}[section]
\newtheorem{Lem}[Thm]{Lemma}

\newtheorem{Prop}[Thm]{Proposition}

\newtheorem{Question}[Thm]{Question}

\theoremstyle{definition}
\newtheorem{Def}[Thm]{Definition}
\newtheorem{Def-Lem}[Thm]{Definition-Lemma}
\newtheorem{Cond}[Thm]{Condition}
\newtheorem{Rem}[Thm]{Remark}
\newtheorem{Ex}[Thm]{Example}

\newtheorem*{Ack}{Acknowledgments}

\theoremstyle{remark}

\newcommand{\Hom}{\operatorname{Hom}}

\newcommand{\prt}{\partial}
\newcommand{\rank}{\operatorname{rank}}

\newcommand{\Sing}{\operatorname{Sing}}
\newcommand{\Spec}{\operatorname{Spec}}

\newcommand{\NQsm}{\operatorname{NQsm}}
\newcommand{\Cl}{\operatorname{Cl}}

\newcommand{\mbA}{\mathbb{A}}

\newcommand{\mbC}{\mathbb{C}}
\newcommand{\mbF}{\mathbb{F}}

\newcommand{\mbP}{\mathbb{P}}
\newcommand{\mbQ}{\mathbb{Q}}
\newcommand{\mbR}{\mathbb{R}}

\newcommand{\mbZ}{\mathbb{Z}}

\newcommand{\mcF}{\mathcal{F}}

\newcommand{\mcI}{\mathcal{I}}

\newcommand{\mcL}{\mathcal{L}}
\newcommand{\mcM}{\mathcal{M}}
\newcommand{\mcN}{\mathcal{N}}
\newcommand{\mcO}{\mathcal{O}}

\newcommand{\mcQ}{\mathcal{Q}}

\newcommand{\mcW}{\mathcal{W}}
\newcommand{\mcY}{\mathcal{Y}}
\newcommand{\mcX}{\mathcal{X}}

\newcommand{\mfm}{\mathfrak{m}}

\newcommand{\msM}{\mathsf{M}}

\newcommand{\msp}{\mathsf{p}}
\newcommand{\msq}{\mathsf{q}}
\newcommand{\K}{\Bbbk}

\newcommand{\inj}{\hookrightarrow}

\newcommand{\ratmap}{\dashrightarrow}

\newcommand{\crit}{\operatorname{cr}}

\newcommand{\wt}{\operatorname{wt}}
\newcommand{\CH}{\operatorname{CH}}

\title[Stable rationality of orbifold Fano threefolds]{Stable rationality of orbifold Fano threefold hypersurfaces}
\author[Takuzo Okada]{Takuzo Okada}
\address{Department of Mathematics, Faculty of Science and Engineering\endgraf
Saga University, Saga 840-8502 Japan}
\email{okada@cc.saga-u.ac.jp}
\date{}

\begin{document}

\begin{abstract}
We determine the rationality of very general quasismooth Fano $3$-fold weighted hypersurfaces completely and determine the stable rationality of them except for cubic $3$-folds. 
More precisely we prove that (i) very general Fano $3$-fold weighted hypersurfaces of index $1$ or $2$ are not stably rational except possibly for the cubic threefolds, (ii) among the $27$ families of Fano 3-fold weighted hypersurfaces of index greater than $2$, very general members of specific $7$ families are not stably rational and the remaining $20$ families consists of rational varieties.
\end{abstract}

\maketitle


\section{Introduction}

The aim of this paper is to study (stable) rationality of orbifold Fano $3$-fold hypersurfaces by the combination of the specialization argument of universal $\CH_0$-triviality initiated by Voisin \cite{Voisin}, developed by Colliot-Th\'el\`ene and Pirutka \cite{CTP}, and the reduction modulo $p$ argument by Koll\'ar \cite{Kollar1,Kollar}.
The combination of these two arguments is firstly applied by Totaro \cite{Totaro} in the proof of stable non-rationality of hypersurfaces.
We recall basic notions and some backgrounds briefly.

A projective variety $X$ is {\it rational} if $X$ is birational to the projective space, and $X$ is {\it stably rational} if there exists $m \ge 0$ such that $X \times \mbP^m$ is rational.
Rationality problem, or in other words determining rationality of algebraic varieties, is a one of the fundamental problems in algebraic geometry.
In dimension $3$, the minimal model program reduces this problem to the same problem for Fano 3-folds of Picard number one, del Pezzo fibrations over $\mbP^1$ and conic bundles over rational surfaces.
We focus on Fano 3-folds (of Picard number one).
Smooth Fano 3-folds have been well studied in this direction and their (stable) rationality is determined (cf.\ \cite{HT}).
A Fano 3-fold of Picard number one, which appears as an outcome of the MMP, in general has $\mbQ$-factorial and terminal singularities and it is necessary to study (stable) rationality of such mildly singular Fano 3-folds.

By an {\it orbifold Fano $3$-fold hypersurface}, we mean a Fano $3$-fold with at most terminal singularities embedded in a weighted projective $4$-space as a well formed and quasi-smooth hypersurface (see Definition \ref{def:qsmwf}).
An orbifold Fano $3$-fold hypersurface is $\mbQ$-factorial, has Picard number $1$ and has only isolated cyclic quotient terminal singularities.
For an orbifold Fano 3-fold hypersurface $X = X_d \subset \mbP (a_0,\dots,a_4)$, where the subscript $d$ indicates the degree of the defining equation of $X$, the positive integer $\alpha$ such that $\mcO_X (-K_X) \cong \mcO_X (\alpha)$ is called the {\it index} (or {\it Fano index}) of $X$.
Explicitly, the index is given by $\alpha = a_0 + \cdots + a_4 - d$. 
There are $130$ families of orbifold Fano $3$-fold hypersurfaces (see \cite{Fletcher}, \cite{BS1}, \cite{BS2}).
Among them, $95$ families consist of index $1$ Fano $3$-folds.
We recall known results on (stable) rationality for orbifold Fano $3$-fold hypersurfaces.

Rationality questions for orbifold Fano $3$-fold hypersurfaces of index $1$ are settled in \cite{IM,CPR,CP} where it is proved that they are birationally rigid, and in particular nonrational.
Among them there are quite a few varieties whose stable non-rationality is known, namely, a very general quartic $3$-fold \cite{CTP} and a very general hypersurface of degree $6$ in $\mbP (1,1,1,1,3)$, which is a double cover of $\mbP^3$ branched along a very general hypersurface of degree $6$ \cite{Beauville}.

Rationality questions for orbifold Fano $3$-fold hypersurfaces of index greater than $1$ have not been settled yet.
It is well known that cubic $3$-folds are not rational \cite{CG}.
Stable non-rationality is known for a very general hypersurface of degree $4$ in $\mbP (1,1,1,1,2)$, which is a double cover of $\mbP^3$ branched along a very general hypersurface of degree $4$ \cite{Voisin}, and for a very general hypersurface of degree $6$ in $\mbP (1,1,1,2,3)$ \cite{HT}.
The above Fano 3-folds are all smooth.
In \cite{Okada1}, rationality questions for weighted hypersurfaces are studied and it is in particular proved that there are two families of singular orbifold Fano $3$-fold hypersurfaces of index $> 1$ whose very general members are not rational.

We state the main theorem of this paper, which completely settles rationality questions for very general orbifold Fano $3$-fold hypersurfaces, and also settles stable rationality questions for them except for cubic $3$-folds.

\begin{Thm} \label{mainthm1}
Suppose that the ground field is the complex number field.
\begin{enumerate}
\item A very general orbifold Fano $3$-fold hypersurface of index $1$ is not stably rational.
\item A very general orbifold Fano $3$-fold hypersurface of index $2$ is not stably rational except possibly for cubic $3$-folds.
\item  Among the $27$ families of Fano $3$-fold hypersurfaces of index greater than $2$, $20$ families consist of rational varieties and a very general member of the remaining $7$ families is not stably rational $($see \emph{Table \ref{qswh}}$)$.
\end{enumerate} 
\end{Thm}

\begin{table}[t] \label{qswh} 
\caption{(Stable) Rationality of orbifold Fano $3$-folds of index $> 1$: 
In the column ``Rat", the signs $+$, $-$ and $--$ mean that a very general member is rational, not rational and not stably rational, respectively.
The column ``Ind" indicates the index of members of the family.}
\begin{tabular}[]{cccc|cccc}
\hline
No. & $X_d \subset \mbP (a_0,\dots,a_4)$ & Rat & Ind & No. & $X_d \subset \mbP (a_0,\dots,a_4)$ & Rat & Ind \\
\hline
96 & $X_3 \subset \mbP (1,1,1,1,1)$ & $-$ & 2 & 113 & $X_4 \subset \mbP (1,1,2,2,3)$ & $+$ & 5 \\[0.2mm]
97 & $X_4 \subset \mbP (1,1,1,1,2)$ & $--$ & 2 & 114 & $X_6 \subset \mbP (1,1,2,3,4)$ & $+$ & 5 \\[0.2mm]
98 & $X_6 \subset \mbP (1,1,1,2,3)$ & $--$ & 2 & 115 & $X_6 \subset \mbP (1,2,2,3,3)$ & $+$ & 5 \\[0.2mm]
99 & $X_{10} \subset \mbP (1,1,2,3,5)$ & $--$ & 2 & 116 & $X_{10} \subset \mbP (1,2,3,4,5)$ & $--$ & 5 \\[0.2mm]
100 & $X_{18} \subset \mbP (1,2,3,5,9)$ & $--$ & 2 & 117 & $X_{15} \subset \mbP (1,3,4,5,7)$ & $--$ & 5 \\[0.2mm] 
101 & $X_{22} \subset \mbP (1,2,3,7,11)$ & $--$ & 2 & 118 & $X_6 \subset \mbP (1,1,2,3,5)$ & $+$ & 6 \\[0.2mm]
102 & $X_{26} \subset \mbP (1,2,5,7,13)$ & $--$ & 2 & 119 & $X_6 \subset \mbP (1,2,2,3,5)$ & $+$ & 7 \\[0.2mm]
103 & $X_{38} \subset \mbP (2,3,5,11,19)$ & $--$ & 2 & 120 & $X_6 \subset \mbP (1,2,3,3,4)$ & $+$ & 7 \\[0.2mm]
104 & $X_2 \subset \mbP (1,1,1,1,1)$ & $+$ & 3 & 121 & $X_8 \subset \mbP (1,2,3,4,5)$ & $+$ & 7 \\[0.2mm]
105 & $X_3 \subset \mbP (1,1,1,1,2)$ & $+$ & 3 & 122 & $X_{14} \subset \mbP (2,3,4,5,7)$ & $--$ & 7 \\[0.2mm]
106 & $X_4 \subset \mbP (1,1,1,2,2)$ & $+$ & 3 & 123 & $X_6 \subset \mbP (1,2,3,3,5)$ & $+$ & 8 \\[0.2mm]
107 & $X_6 \subset \mbP (1,1,2,2,3)$ & $--$ & 3 & 124 & $X_{10} \subset \mbP (1,2,3,5,7)$ & $+$ & 8 \\[0.2mm]
108 & $X_{12} \subset \mbP (1,2,3,4,5)$ & $--$ & 3 & 125 & $X_{12} \subset \mbP (1,3,4,5,7)$ & $+$ & 8 \\[0.2mm]
109 & $X_{15} \subset \mbP (1,2,3,5,7)$ & $--$ & 3 & 126 & $X_6 \subset \mbP (1,2,3,4,5)$ & $+$ & 9 \\[0.2mm]
110 & $X_{21} \subset \mbP (1,3,5,7,8)$ & $--$ & 3 & 127 & $X_{12} \subset \mbP (2,3,4,5,7)$ & $+$ & 9 \\[0.2mm]
111 & $X_4 \subset \mbP (1,1,1,2,3)$ & $+$ & 4 & 128 & $X_{12} \subset \mbP (1,4,5,6,7)$ & $+$ & 11 \\[0.2mm]
112 & $X_6 \subset \mbP (1,1,2,3,3)$ & $+$ & 4 & 129 & $X_{10} \subset \mbP (2,3,4,5,7)$ & $+$ & 11 \\[0.2mm]
& & & & 130 & $X_{12} \subset \mbP (3,4,5,6,7)$ & $+$ & 13 \\
\end{tabular}
\end{table}

We can re-state Theorem \ref{mainthm1} in the following way, which gives a simple characterization of (stable) rationality of very general orbifold Fano $3$-fold hypersurfaces in terms of weights of the ambient space and the degree of the hypersurface.

\begin{Thm} \label{mainthm2}
Let $X = X_d \subset \mbP (a_0,\dots,a_4)$, $a_0 \le \cdots \le a_4$, be a very general orbifold Fano $3$-fold hypersurface of degree $d$ defined over the complex number field.
Then the following are equivalent.
\begin{enumerate}
\item Either $d < 2 a_4$ or $d = 2 a_4 = 2 a_3$.
\item $X$ is rational.
\end{enumerate}
Moreover, if $X$ is not a cubic $3$-fold, then the above $2$ conditions are equivalent to the following.
\begin{enumerate}
\item[(3)] $X$ is stably rational.
\end{enumerate}
\end{Thm}

Note that the implication $(1) \Rightarrow (2)$ is easy (see Section \ref{sec:rat}) and $(2) \Rightarrow (3)$ is trivial.
The main result of this paper is to prove the implication $(3) \Rightarrow (1)$.

The implication $(1) \Rightarrow (2)$ holds true in any dimension: for a general orbifold Fano hypersurface $X = X_d \subset \mbP (a_0,\dots,a_{n+1})$, $a_0 \le \dots \le a_{n+1}$, of degree $d$, if either $d < 2 a_{n+1}$ or $d = 2 a_{n+1} = 2 a_n$, then $X$ is rational.
The following question arises naturally.

\begin{Question}
Let $X = X_d \subset \mbP (a_0,\dots,a_{n+1})$, $a_0 \le \cdots \le a_{n+1}$, be a very general orbifold Fano hypersurface of degree $d$ and dimension $n \ge 3$.
Is there any $X$ which is rational but satisfies neither $d < 2 a_{n+1}$ nor $d = 2 a_n = 2 a_{n+1}$?
Moreover is there any $X$ which is stably rational but not rational?
\end{Question}

We explain a rough sketch of the proof of main theorems and then the organization of the paper.
To each family of orbifold Fano $3$-fold hypersurfaces which do not satisfy (1) of Theorem \ref{mainthm2}, we construct a subfamily  whose members admit a cyclic covering structure over a weighted hypersurface.
We then consider the subfamily over an algebraically closed field of characteristic $p$, where $p$ is a prime number dividing the covering degree, so that a member $X$ of the subfamily is an inseparable covering of a weighted hypersurface.
By the Koll\'ar's argument, we can prove that there exists a non-zero global differential $2$-form $\eta$ on $X$ (which is regular on the smooth locus of $X$).
The next task is to construct a resolution $\varphi \colon Y \to X$ of singularities of $X$ satisfying good properties.
Here good properties mean that $\varphi$ is universally $\CH_0$-trivial and $\varphi^* \eta$ is a regular form on $Y$.
The latter implies that $Y$ is not universally $\CH_0$-trivial.
Now we lift $X$ to an orbifold Fano $3$-fold hypersurface over $\mbC$, and, by the specialization property of universal $\CH_0$-triviality, we can conclude that a very general member of the considered family is not stably rational.

In Section \ref{sec:framework}, we explain in detail that the existence of the above mentioned subfamily indeed implies the stable non-rationality of a very general member of the family.
In Section \ref{sec:resol}, we consider weighted hypersurfaces $X$ admitting an inseparable cyclic covering structure over a weighted hypersurface $Z$ and give a condition for $X$ to admit a resolution of singularities $\varphi \colon Y \to X$ satisfying good properties.
The most important condition is the mildness of singularities of $X$, and thus we need to control them.
We study singularities of $X$ in terms of quasi-smoothness of $X$ along a suitable stratum of the ambient space and in terms of critical points of the section defining the branched divisor of the covering $X \to Z$.
In Section \ref{sec:qsm}, we give quasi-smoothness criteria for weighted hypersurfaces in positive characteristic and in Section \ref{sec:critical} we give a criterion for a suitable section on a weighted hypersurface to have only mild critical points.
In Section \ref{sec:rat} we consider rationality of orbifold Fano $3$-fold hypersurfaces.
In Section \ref{sec:stnonrat} we apply criteria in Sections \ref{sec:qsm} and \ref{sec:critical} for orbifold Fano $3$-fold hypersurfaces and show that the condition given in Section \ref{sec:resol} is satisfied, which will complete the proof of stable non-rationality by the result of Section \ref{sec:framework}.
Finally, in Section \ref{sec:stnonrat}, we exhibit an example of a stably non-rational orbifold Fano $3$-fold hypersurface obtained in this paper and show that the rationality criterion \cite[Theorem 1.8]{BMSZ} in terms of absolute complexity is sharp.

\begin{Ack}
The author would like to thank Professor Ivan Cheltsov for having interest on this work.
The author is partially supported by JSPS KAKENHI Grant Number 26800019.
\end{Ack}

\section{Preliminaries} \label{sec:framework}

\subsection{Universal $\CH_0$-triviality}

We explain the definition and basic properties of universally $\CH_0$-triviality.
For a variety $X$, we denote by $\CH_0 (X)$ the Chow group of $0$-cycles on $X$, which is by definition the free abelian group of $0$-cycles modulo rational equivalence.

\begin{Def}
\begin{enumerate}
\item A projective variety $X$ defined over a field $k$ is {\it universally $\CH_0$-trivial} if for any field $F$ containing $k$, the degree map $\CH_0 (X_F) \to \mbZ$ is an isomorphism.
\item A projective morphism $\varphi \colon Y \to X$ defined over a field $k$ is {\it universally $\CH_0$-trivial} if for any field $F$ containing $k$, the push-forward map $\varphi_* \colon \CH_0 (Y_F) \to \CH_0 (X_F)$ is an isomorphism.
\end{enumerate}
\end{Def}

We apply the following specialization arguments to orbifold Fano $3$-fold hypersurfaces.

\begin{Lem}[{\cite[Lemma 1.5]{CTP}}]
If $X$ is a smooth, projective, stably rational variety, then $X$ is universally $\CH_0$-trivial.
\end{Lem}

\begin{Thm}[{\cite[Th\'eor\`eme 1.14]{CTP}}] \label{thm:sp}
Let $A$ be a discrete valuation ring with fraction field $K$ and residue field $k$, with $k$ algebraically closed.
Let $\mcX$ be a flat proper scheme over $A$ with geometrically integral fibers.
Let $X$ be the generic fiber $\mcX \times_A K$ and $Y$ the special fiber $\mcX \times_A k$.
Assume that $Y$ admits a universally $\CH_0$-trivial resolution $\tilde{Y} \to Y$ of singularities.
Let $\overline{K}$ be an algebraic closure of $K$ and assume that the geometric generic fiber $X_{\overline{K}}$ admits a universally $\CH_0$-trivial resolution $\tilde{X} \to X_{\overline{K}}$.
If $\tilde{X}$ is universally $\CH_0$-trivial, then so is $\tilde{Y}$.
\end{Thm}

The following is a sufficient condition for universally $\CH_0$-non-triviality.

\begin{Lem}[{\cite[Lemma 2.2]{Totaro}}] \label{lem:totaro}
Let $X$ be a smooth projective variety over a field $k$.
If $H^0 (X,\Omega_X^i) \ne 0$ for some $i > 0$, then $X$ is not universally $\CH_0$-trivial.
\end{Lem}

\subsection{Framework of proof}

Let $a_0,\dots,a_{n+1}, d$ be positive integers and $\mbP (a_0,\dots,a_{n+1})$ the weighted projective space with homogeneous coordinates $x_0,\dots,x_{n+1}$.
Let $\mcX \to \mbP_{\mbZ}^M$ be the family of weighted hypersurfaces of degree $d$ in $\mbP (a_0,\dots,a_{n+1})$ defined over $\mbZ$.
Here $\mbP^M_{\mbZ}$ parametrizes the polynomials of degree $d$ with coefficients in $\mbZ$ and in variables $x_0,\dots,x_{n+1}$.
For a field or a ring $R$, we denote by $\mcX_R \to \mbP^M_R$ the base change of $\mcX \to \mbP^M_{\mbZ}$, which the family of weighted hypersurfaces of degree $d$ in $\mbP (a_0,\dots,a_{n+1})$ defined over $R$.

Our aim is to construct a (locally closed) subspace $T \cong \mbA^N_{\mbZ}$ of $\mbP^M_{\mbZ}$, $0 < N \le M$,  satisfying the following condition.
For a field $k$, we define $T_k = T \times_{\Spec \mbZ} \Spec k\subset \mbP^M_k$.

\begin{Cond} \label{cd:T}
\begin{enumerate}
\item A general member of the subfamily parametrized by $T_{\mbC} \subset \mbP^M_{\mbC}$ is quasi-smooth and has only isolated cyclic quotient singularities.
\item There exists an algebraically closed field $\K$ of positive characteristic $p$ such that $T_{\K}^{\operatorname{indep}} \ne \emptyset$ (see Definition \ref{def:indep} below) and a very general member $X'$ of the subfamily parametrized by $T_{\K} \subset \mbP^M_{\K}$ has only isolated singular points and admits a resolution $\varphi' \colon Y' \to X'$ of singularities with the following properties:
\begin{enumerate}
\item[(a)] $\varphi'$ is an isomorphism over the smooth locus of $X'$ and the exceptional divisor of $\varphi'$ is a simple normal crossing (abbreviated as SNC) divisor whose components are smooth rational varieties.
\item[(b)] $H^0 (Y',\Omega_{Y'}^{n-1}) \ne 0$.
\end{enumerate}
\end{enumerate}
\end{Cond}

\begin{Def} \label{def:indep}
For a field $k$, we define $T_k^{\operatorname{indep}}$ to be the subset of $T_k = \mbA^N_k$ consisting of the point $(\alpha_1,\dots,\alpha_N) \in \mbA^N_k$ such that $\alpha_1, \dots,\alpha_N$ are algebraically independent over the prime field of $k$.
\end{Def}

\begin{Ex}
We explain by an example that what kind of $T$ we will consider.
Let us consider weighted hypersurfaces of degree $9$ in $\mbP = \mbP (1,1,1,3,4)$.
Let $x,y,z,w,t$ be the homogeneous coordinates of weight $1,1,1,3,4$, respectively.
The polynomials (up to a multiple of non-zero constant) in $x,y,z,w,t$ of degree $9$ with coefficients in $\mbZ$ can be parametrized by $\mbP^{102}_{\mbZ}$, so that we have a family $\mcX \to \mbP^{102}_{\mbZ}$ of weighted hypersurfaces of degree $9$ in $\mbP$ defined over $\mbZ$.
We consider the subfamily consisting of hypersurfaces defined by an equation of the form
\[
w^3 + f_9 (x,y,z,t) = 0.
\]
Let $T$ be the affine space parametrizing degree $9$ polynomials $f_9 (x,y,z,t)$ in variables $x,y,z,t$.
We see that $T \cong \mbA_{\mbZ}^{94}$ and we can naturally embed $T \inj \mbP^{102}_{\mbZ}$ so that the fiber of $\mcX \to \mbP^M$ over points of $T$ are hypersurfaces defined by $w^3 + f_9 = 0$.
The members parametrized by $T$ are cyclic covers of $\mbP (1,1,1,3)$ branched along a divisor of degree $9$.
\end{Ex}

The most crucial condition is (2) whose verifications for orbifold Fano $3$-fold hypersurfaces will be done in Section \ref{sec:stnonrat}.
In this section, we explain that the existence of $T$ satisfying Condition \ref{cd:T} implies that a very general member of $\mcX_{\mbC} \to \mbP^M_{\mbC}$ is not stably rational.
Although the arguments below may be well known to experts, we include them for readers' convenience. 

We keep the above setting and let $\mcY = \mcX \times_{\Spec \mbZ} T \to T$ be the subfamily of $\mcX \to \mbP^M_{\mbZ}$. 
Let $\K$ be an algebraically closed field of characteristic $p$ as in Condition \ref{cd:T}.(2).

\begin{Rem} \label{rem:sing}
\begin{enumerate}
\item By Condition \ref{cd:T}.(1), a general member of the family $\mcX_k \to \mbP^M_k$ is quasi-smooth and has only isolated cyclic quotient singularities for any algebraically closed field $k$ of characteristic $0$.
\item An isolated cyclic quotient singularity (defined over an algebraically closed field) admits a resolution of singularities whose exceptional divisor is a simple normal crossing divisor and each component is a nonsingular rational (toric) variety (see \cite[Theorem 11.2.2]{CLS}).
\item By \cite[Proposition 1.8]{CTP} and \cite[Lemma 2.4]{CTP2}, a resolution $\varphi' \colon Y' \to X'$ whose exceptional divisor is a SNC divisor with smooth rational components is universally $\CH_0$-trivial.
\end{enumerate}
\end{Rem}

%

Note that the set $T_k \setminus T_k^{\operatorname{indep}}$ is a countable union of divisors and hence $T_k^{\operatorname{indep}}$ is non-empty if $k$ is uncountable.

\begin{Lem} \label{lem:T1}
Let $X$ be a very general member of the family $\mcY_{\mbC} \to T_{\mbC}$.
Then there exists a universally $\CH_0$-trivial resolution $Y \to X$ of singularities such that $Y$ is not universally $\CH_0$-trivial.
\end{Lem}

\begin{proof}
%
We may assume that $\K$ is countable.
Indeed, we can take finitely many elements $\gamma_1,\dots,\gamma_m \in \K$ such that, for $k = \mbF_p (\gamma_1,\dots,\gamma_m) \subset \K$, every objects appearing in Condition \ref{cd:T}.(2) ($X'$, $\varphi' \colon Y' \to X'$, etc.)  can be defined over the algebraic closure $\bar{k} \subset \K$ and $T_{\bar{k}}^{\operatorname{indep}} \ne \emptyset$.
Replacing $\K$ with $\bar{k}$, we may assume that $\K$ is countable.

Let $R = W (\K)$ be the ring of Witt vectors over $\K$, which is a complete discrete valuation ring whose residue field is $\K$ and the quotient field $K$ is of characteristic $0$.
Since $R = \K \oplus \K \oplus \cdots$ set-theoretically, its quotient field $K$ is countable.
This implies that there is an embedding $\iota_0 \colon \bar{K} \inj \mbC$, where $\bar{K}$ is a fixed algebraic closure of $K$.

Let $X$ be a very general member of the family $\mcY_{\mbC} \to T_{\mbC}$ and $P = (\alpha_1,\dots,\alpha_N) \in T_{\mbC}^{\operatorname{indep}}$ the corresponding point.
We choose and fix a point $P' = (\alpha'_1,\dots,\alpha'_N) \in T_{\K}^{\operatorname{indep}}$ and let $X'$ be the corresponding member of $\mcY_{\K} \to T_{\K}$.
For each $i$, we choose and fix a lift $a_i \in R$ of $\alpha'_i$ via $R \to \K$.
Let $V$ be the fiber of $\mcX_R \to T_R$ over the $R$-point $(a_1,\dots,a_N) \in T_R$.
Note that $V$ is a projective scheme over $R$ whose special fiber $V_{\K}$ is isomorphic to $X'$.
By Condition \ref{cd:T}.(2) (and see also Remark \ref{rem:sing}), $V_{\K} \cong X'$ admits a universally $\CH_0$-trivial resolution $\varphi' \colon Y' \to X'$. 
Moreover $Y'$ is not universally $\CH_0$-trivial by Lemma \ref{lem:totaro}. 
Since the $\alpha'_i$ are algebraically independent over $\mbF_p$, the $a_i \in K$ are algebraically independent over $\mbQ$.
It follows that the geometric generic fiber $V_{\bar{K}}$ is a very general member of the family $\mcY_{\overline{K}} \to T_{\overline{K}}$.
In particular it is quasi-smooth and has only isolated cyclic quotient singularities. 
Thus there exists a universally $\CH_0$-trivial resolution $\tilde{V}_{\bar{K}} \to V_{\bar{K}}$ of singularities (see Remark \ref{rem:sing}).
Hence, by Theorem \ref{thm:sp}, $\tilde{V}_{\bar{K}}$ is not universally $\CH_0$-trivial.
Now we can choose an automorphism $\tau \colon \mbC \to \mbC$ which maps $\iota_0 (a_i)$ to $\alpha_i$. 
We set $\iota = \tau \circ \iota_0 \colon \bar{K} \inj \mbC$.
Then the base change via $\iota \colon K \inj \mbC$ of the generic fiber of $V \to \Spec R$ is isomorphic to $X$ and the base change $\tilde{V}_{\mbC} \to V_{\mbC} \cong X$ via $\iota$ of the resolution $\tilde{V}_{\bar{K}} \to V_{\bar{K}}$ gives a universally $\CH_0$-trivial resolution of $X$.
The proof is completed since $\tilde{V}_{\mbC}$ is not universally $\CH_0$-trivial.
\end{proof}

\begin{Lem}
A very general member of the family $\mcX_{\mbC} \to \mbP^M_{\mbC}$ is not stably rational.
\end{Lem}

\begin{proof}
Let $\bar{K}$ be an algebraic closure of the function field $K = \mbC (\mbP^M_{\mbC})$ and let $\mcX_{\bar{K}}$ be the geometric generic fiber of $\mcX \to \mbP^M_{\mbC}$.
For a closed point $P \in \mbP^M_{\mbC}$, we denote by $X_P$ the fiber of $\mcX \to \mbP^M_{\mbC}$ over $P$.
By \cite[Lemma 2.1]{Vial}, there exists a subset $\Sigma \subset \mbP^M_{\mbC}$ which is a countable union of proper closed subsets of $\mbP^M_{\mbC}$ such that $X_P$ is isomorphic to $\mcX_{\bar{K}}$ as an abstract scheme.
The variety $\mcX_{\bar{K}}$ has only isolated cyclic quotient singularities and thus admits a universally $\CH_0$-trivial toric resolution $\tilde{X}_{\bar{K}} \to \mcX_{\bar{K}}$.
Moreover, if we are given a point $P \in \mbP^M_{\mbC} \setminus \Sigma$, then the fiber $X_P$ admits a universally $\CH_0$-trivial resolution $Y_P \to X_P$ such that $Y_P$ is isomorphic to $\tilde{X}_{\bar{K}}$ as an abstract scheme.
Since the Chow group of a variety $X$ only depend on $X$ as a scheme (see \cite[Lemma 2.1]{Vial}), it follows that $Y_P$ is universally $\CH_0$-trivial if and only if so is $\tilde{X}_{\bar{K}}$.
Thus, if we show that there exists a point $P \in \mbP^M_{\mbC} \setminus \Sigma$ such that $Y_P$ is not universally $\CH_0$-trivial, then, for any $P' \in \mbP^M_{\mbC} \setminus \Sigma$, $Y_{P'}$ is not universally $\CH_0$-trivial, hence $X_{P'}$ is not stably rational.

By Lemma \ref{lem:T1}, for a very general point $Q \in T_{\mbC} \subset \mbP^M_{\mbC}$, the fiber $X_Q$ of $\mcY \to T_{\mbC}$ over $Q$ admits a universally $\CH_0$-trivial resolution $Y_Q \to X_Q$ of singularities such that $Y_Q$ is not universally $\CH_0$-trivial.
Let $C \subset \mbP^M_{\mbC}$ be a nonsingular curve such that $Q \in C$ and $C \not\subset \Sigma$.
We can indeed take such a curve $C$ by choosing any point $P'' \in \mbP^M_{\mbC} \setminus \Sigma$ and successively cutting down $\mbP^M_{\mbC}$ by general hyperplanes passing through $P''$ and $Q$.
By Theorem \ref{thm:sp} applied to the local ring $\mcO_{C,Q}$, the geometric generic fiber $\mcX_{\overline{\mbC (C)}}$ of $\mcX \times_{\mbP^M_{\mbC}} C \to C$ admits a universally $\CH_0$-trivial toric resolution $\tilde{\mcX}_{\overline{\mbC (C)}} \to \mcX_{\overline{\mbC} (C)}$ such that $\tilde{\mcX}_{\overline{\mbC (C)}}$ is not universally $\CH_0$-trivial.
Repeating the same argument as in the first part of the proof, we conclude that the fiber $X_P$ admits a universally $\CH_0$-trivial toric resolution $Y_P \to X_P$ such that $Y_P$ is universally $\CH_0$-trivial for a very general point $P \in C$.
Since $P \in C$ is very general and $\Sigma \cap C \ne \emptyset$, we may assume that $P \notin \Sigma$.
Therefore the proof is completed.
\end{proof}

\section{General construction of a good resolution} \label{sec:resol}

\subsection{Cyclic covers and admissible critical points}

We briefly recall Koll\'ar's construction of a suitable line bundle on an inseparable cyclic covering space and then give definition of critical points (see \cite[Section V.5]{Kollar} for details).

Let $Z$ be a smooth variety of dimension $n$ over an algebraically closed field of positive characteristic $p$, $\mcL$ a line bundle on $Z$, $m$ a positive integer and $s \in H^0 (Z,\mcL^m)$ a global section.
Let $\pi \colon X \to Z$ be the cyclic cover of degree $m$ branched along the zero locus $(s = 0) \subset Z$.
Throughout the present section, we assume that $p \mid m$ and that the branched divisor $(s = 0)$ is reduced.
In this setting, there is a line bundle on $\mcQ (\mcL,s)$ on $Z$ such that $\pi^*\mcQ (\mcL,s) \subset (\Omega_X^{n-1})^{\vee \vee}$, where $(\Omega^{n-1}_X)^{\vee \vee}$ denotes the double dual of $\Omega_X^{n-1}$, and $\mcQ (\mcL,s) \cong \omega_Z \otimes \mcL^m$.

Singularities of $X$ can be understood by critical points of $s$.
Let $\msq \in Z$ be a point, $x_1,\dots,x_n$ local coordinates of $Z$ at $\msq$ and $\tau$ a local generator of $\mcL$ at $\msq$.
Then, locally around $\msq$, we can write $s = f \tau^m$, where $f = f (x_1,\dots,x_n) \in \mcO_{Z,\msq}$.

\begin{Def}
We say that $s$ has a {\it critical point} at $\msq$ if $\prt f/\prt x_1 = \cdots = \prt f/\prt x_n = 0$ at $\msq$.
\end{Def}

Note that the above definition does not depend on the choice of a local generator $\tau$ and local coordinates $x_1,\dots,x_n$.
We have
\[
\Sing X = \pi^{-1} (\{ \text{critical point of $s$} \}).
\]
We give a definition of admissible critical point of $s$, which ensures some mildness of singularities of $X$.
The following definition is complicated and we refer readers to \cite[Section 3.3]{Okada3} for details.

\begin{Def}
We say that $s \in H^0 (Z,\mcL)$ has an {\it admissible critical point} at $\msq \in Z$ if in a local expression $s = f \tau^m$, $f$ satisfies one of the following:
\begin{enumerate}
\item Either $n$ is even or $n$ is odd and $p \ne 2$, and the quadratic part of $f$ is nondegenerate.
\item $n$ is odd, $p = 2$, $m = 2$ and $\operatorname{length} \mcO_{Z,\msq}/(\prt f/\prt x_1,\dots, \prt f/\prt x_{n}) = 2$.
\item $n$ is odd, $p = 2$, $m \ne 2$, $2^2 \nmid m$, $\operatorname{length} \mcO_{Z,\msq}/(\prt f/\prt x_1,\dots, \prt f/\prt x_{n}) = 2$ and $s$ does not vanish at $\msq$.
\item $n$ is odd, $p = 2$, $m \ne 2$, $2^2 \nmid m$, $\operatorname{length} \mcO_{Z,\msq}/(\prt f/\prt x_1,\dots, \prt f/\prt x_{n}) = 2$, $s$ vanishes at $\msq$ and the quadratic part of $f$ is nondegenerate.
\item $n$ is odd, $p = 2$, $2^2 \mid m$, $\operatorname{length} \mcO_{Z,\msq}/(\prt f/\prt x_1,\dots, \prt f/\prt x_{n}) = 2$ and the quadratic part of $f$ is nondegenerate.
\end{enumerate}
\end{Def}

Note that the above definition does not depend on the choice of $\tau$ and $x_1,\dots,x_n$.

\begin{Rem}
Suppose that $n$ is odd and $p = 2$.
In this case, by \cite[Section V.5]{Kollar}, the condition $\operatorname{length} \mcO_{Z,\msq}/(\prt f/\prt x_1,\dots, \prt f/\prt x_{n}) = 2$ is satisfied if and only if in a suitable choice of local coordinates $x_1,\dots,x_n$, $f$ can be written as
\[
f = \alpha + \beta x_1^2 + x_2 x_3 + x_4 x_5 + \cdots + x_{n-1} x_n + \gamma x_1^3 + g (x_1,\dots,x_n),
\]
where $\alpha, \beta, \gamma \in \K$ with $\gamma \ne 0$ and $g$ is a linear combination of monomials of degree at least $3$ other than $x_1^3$.
Under the above choice of coordinates, $f$ is nondegenerate if and only if $\beta \ne 0$.
\end{Rem}

\subsection{Construction}

Let $\mbP = \mbP (a_0,\dots,a_{n},b)$ be a weighted projective space defined over an algebraically closed field $\K$ of positive characteristic $p$ with homogeneous coordinates $x_0,\dots,x_n$ and $w$ of weight $a_0,\dots,a_n$ and $b$, respectively.
Let $m$ be a positive integer divisible by $p$.
Let $X$ be a weighted hypersurface in $\mbP$ defined by
\[
f (x_0,\dots,x_n,w^m) = 0.
\]
We define $Z$ to be the weighted hypersurface defined by $f (x_0,\dots,x_n,\bar{w}) = 0$ in the weighted projective space $\mbP (a_0,\dots,a_n,mb)$ of coordinates $x_0,\dots,x_n$ and $\bar{w}$ and let $\pi \colon X \to Z$ be the morphism defined by $\pi^* \bar{w} = w^m$.
We define $\mcL = \mcO_Z (b)$.
Then, $\bar{w}$ is a global section of $(\mcL^m)^{\vee \vee} \cong \mcO_Z (mb)$. 
We set $a_{\operatorname{sum}} = \sum_{i=0}^n a_i$. 
We introduce the following condition on $X$ and $Z$.

\begin{Cond} \label{cdgen}
\begin{enumerate}
\item $Z$ is well-formed (see Definition \ref{def:qsmwf}) and normal.
\item There exists a non-empty smooth open subset $Z^{\circ} \subset Z$ such that the section $\bar{w}$ has only admissible critical points on $Z^{\circ}$ and $X$ has at most isolated cyclic quotient singular points along $X \setminus \pi^{-1} (Z^{\circ})$. 
\item $n \ge 3$.
\item $d - a_{\operatorname{sum}} \ge 0$ and $H^0 (Z, \mcO_Z (d-a_{\operatorname{sum}})) \ne 0$.
\end{enumerate}
\end{Cond}

\begin{Prop} \label{prop:constr}
If $X$ satisfies \emph{Condition \ref{cdgen}}, then there exists a resolution $\varphi \colon Y \to X$ of singularities of $X$ such that the exceptional divisor is a SNC divisor with smooth rational components and $H^0 (Y,\Omega_Y^{n-1}) \ne 0$.
\end{Prop}

\begin{proof}
Let $V$ be the smooth locus of $Z$ and set $U = \pi^{-1} (V) \subset X$.
By \cite[Section V.5]{Kollar}, there exists a sub line bundle $\mcM_U := \pi^* \mcQ (\mcL|_V, \bar{w})$ of $(\Omega_{U}^{n-1})^{\vee \vee}$.
Condition \ref{cdgen}.(1) implies that $\omega_Z \cong \mcO_Z (d-a_{\operatorname{sum}} - m b)$ and Condition \ref{cdgen}.(2) in particular implies that the branched divisor $(s = 0) \subset Z$ is reduced.
Hence we have an isomorphism
\[
\mcM_U \cong \pi^* (\omega_V \otimes (\mcL|_V)^{\otimes m}) \cong \mcO_{U} \left( d-a_{\operatorname{sum}}\right).
\]
We define $\mcM \subset (\Omega_X^{n-1})^{\vee \vee}$ to be the pushforward of $\mcM_U$ by the injection $U \inj X$.
Note that $\mcM \cong \mcO_X (d-a_{\operatorname{sum}})$ and it is a reflexive sheaf of rank $1$ (which may not be an invertible sheaf in general).

Let $t$ be any global section of $\mcM \cong \mcO_Z (d-a_{\operatorname{sum}})$, which exists by Condition \ref{cdgen}.(4).
We have an injection $\mcO_X \inj \mcM$, which is a multiplication by $t$, and let $\mcN \cong \mcO_X$ be its image.

Note that $s$ does not have a critical point at $V \setminus Z^{\circ}$ because otherwise $X$ has a non-quotient singular point along $X \setminus \pi^{-1} (Z^{\circ})$ which is impossible by Condition \ref{cdgen}.(2).
It follows that $s$ has only admissible critical points on $V$.
Thus, by Conditions \ref{cdgen}.(3) and \cite[Proposition 4.1]{Okada3}, there exists a resolution $\varphi_U \colon Y_U \to U$ of singularities of $U$ such that the exceptional divisor is a SNC divisor with smooth rational components and ${\varphi_U}^* (\mcM|_U) \inj \Omega_{Y_U}^{n-1}$.
This implies that $\varphi_U^*(\mcN|_U) \inj \Omega_Y^{n-1}$.
Let $\varphi \colon Y \to X$ be a resolution such that $\varphi$ coincides $\varphi_U$ over $U$ and $\varphi$ is a toric resolution of singularities of isolated cyclic quotient singular points on $X \setminus X^{\circ}$ such that the fiber of $\varphi$ over any cyclic quotient singular point is a SNC divisor whose component is a nonsingular toric variety.
By Lemma \ref{lem:toriclift} below (see also Remark \ref{rem:lift}), we conclude that $\mcO_Y \cong \varphi^*\mcN \inj \Omega_Y^{n-1}$.
Therefore $H^0 (Y,\Omega_Y^{n-1}) \ne 0$.
\end{proof}

\subsection{Lifting lemma for differential forms on toric varieties}

In this subsection, we prove that the pullback via a toric resolution of a differential $j$-form on a toric variety is a regular $j$-form.

We recall necessary definitions of toric varieties and we refer readers to \cite[Section 4]{Danilov} for details.
Let $\msM$ be an $n$-dimensional lattice and $\sigma \subset \msM$ a convex rational polyhedral cone generating $\msM_{\mbQ} = \msM \otimes_{\mbZ} \mbQ$.
Let $k$ be a field and we set $A = k [\sigma \cap \msM]$, $X = \Spec A$.
For $m \in \sigma \cap \msM$, we denote by $\chi^m \in A$ the corresponding monomial, and by $\Gamma_{\sigma} (m)$ the smallest face of $\sigma$ containing $m$.

We set $V = \msM \otimes_{\mbZ} k$.
For a face $\tau \subset \sigma$, we define a subspace $V_{\tau} \subset V$ as follow:
if $\tau$ is of codimension one, then we define
\[
V_{\tau} = (\msM \cap (\tau - \tau)) \otimes_{\mbZ} k
\]
and in general we define
\[
V_{\tau} = \bigcap_{\theta \supset \tau} V_{\theta},
\]
where $\theta$ ranges over the faces of $\tau$ of codimension $1$ containing $\tau$. 
For $j = 1,2,\dots,n$, we define 
\[
\Omega_{\sigma}^j = \bigoplus_{m \in \sigma \cap \msM} \bigwedge^j (V_{\Gamma (m)}) \cdot \chi^m,
\]
which is a $\msM$-graded $k$-vector space.
It is easy to see that $\Omega^j_{\sigma}$ is naturally embedded into the $A$-module $(\bigwedge^j V) \otimes_k A$ and thus equipped with the structure of an $\msM$-graded $A$-module.

\begin{Prop}[Proposition 4.3, \cite{Danilov}] \label{propdanilov}
The sheaf $(\Omega_X^j)^{\vee \vee}$ is isomorphic to the sheaf associated with the $A$-module $\Omega_{\sigma}^j$. 
\end{Prop}

\begin{Lem} \label{lem:toriclift}
Let $X$ be a toric variety over an algebraically closed field $k$ and $\varphi \colon Y \to X$ a toric resolution of singularities of $X$.
Then there is a homomorphism $\varphi^* ((\Omega_X^j)^{\vee \vee}) \to \Omega_Y^j$ factoring $\varphi^* \Omega_X^j \to \Omega_Y^j$ for every $j = 1,\dots,\dim X$.
\end{Lem}

\begin{proof}
A toric resolution $\varphi \colon Y \to X$ is obtained by subdividing the fan (in $\Hom_{\mbZ} (\msM,\mbZ)$) which defines $X$.
We may assume that both $X$ and $Y$ are affine toric varieties since this is a local problem.

Let $X = \Spec A$ and $A = k [\sigma \cap \msM]$, where $\msM$ is a lattice and $\sigma$ is a cone in $\msM_{\mbR}$ generating $\msM_{\mbR}$.
Then we may assume that $Y = \Spec k [\sigma' \cap \msM]$, where $\sigma'$ is a cone in $\msM_{\mbR}$ such that $\sigma' \supset \sigma$.
It suffices to show that $V_{\Gamma_{\sigma} (m)} \subset V_{\Gamma_{\sigma'} (m)}$ for every $m \in \sigma \cap \msM$.
Indeed, then, there is a natural homomorphism of $A' = k [\sigma' \cap \msM]$-modules
\[
\Omega^j_{\sigma} \otimes_A A' \to \Omega^j_{\sigma'},
\]
which, together with Proposition \ref{propdanilov} shows that there is a homomorphism $\varphi^* ((\Omega^j_X)^{\vee \vee}) \to \Omega^j_Y$ factoring $\varphi^*\Omega^j_X \to \Omega^j_Y$.

First, suppose that $m$ is contained contained in the interior of $\sigma \cap \msM$.
Then, $m$ is contained in the interior of $\sigma' \cap \msM$.
In this case we have $V_{\Gamma_{\sigma} (m)} = V_{\Gamma_{\sigma'} (m)}$ and they coincide with $V = \msM \otimes_{\mbZ} k$.
Suppose next that $m$ is contained in the boundary of $\sigma \cap \msM$.
Let $\tau$ be a codimension one face of $\sigma$ which contains $m$.
If $\tau$ is not contained in a face of $\sigma'$ then $m$ is contained in the interior of $\sigma' \cap \msM$.
Now recall that $V_{\Gamma_{\sigma} (m)}$ is the intersection of $V_{\tau}$, where $\tau$ runs over the codimension one faces of $\sigma$ which contain $m$.
Therefore we have $V_{\Gamma_{\sigma} (m)} \subset V_{\Gamma_{\sigma'} (m)}$, and the proof is completed.
\end{proof}

\begin{Rem} \label{rem:lift}
Let $x \in X$ be a germ of an isolated toric singularity and let $\varphi \colon Y \to X$ be a toric resolution of $x \in X$.
Lemma \ref{lem:toriclift} implies that the pullback via $\varphi$ of any differential $j$-form $\eta \in (\Omega_X^j)^{\vee \vee}$, viewed as a rational $j$-form, is a regular $j$-form on $Y$.
In particular, for a line bundle $\mcL \subset (\Omega_X^j)^{\vee \vee}$, we have $\varphi^* \mcL \subset \Omega_Y^j$.
\end{Rem}

\section{Quasi-smoothness in positive characteristic} \label{sec:qsm}

A simple characterization of quasi-smoothness of weighted complete intersections defined over an algebraically closed field of characteristic $0$ is given by Iano-Fletcher \cite{Fletcher} (see also \cite{Okada2} for a slight generalization), which is based on Bertini theorem.
The aim of this section is to give a quasi-smoothness criterion for weighted hypersurfaces in positive characteristics.
Although our argument is technically involved, it is primitive and avoids the use of Bertini theorem.


We introduce basic definitions.
Let $\mbP = \mbP (a_0,\dots,a_n)$ be a weighted projective space defined over an algebraically closed field $k$ with homogeneous coordinates $x_0,\dots,x_n$ of weight $a_0,\dots,a_n$, respectively.
We always assume that $\mbP$ is {\it well-formed}, that is, 
\[
\gcd (a_0,\dots,\hat{a}_i,\dots,a_n) = 1
\] 
for any $i$.

\begin{Def} \label{def:qsmwf}
Let $X$ be a closed subscheme of $\mbP$ and $\tau \colon \mbA^{n+1} \setminus \{o\} \to \mbP$ the natural projection.

We say that $X$ is {\it quasi-smooth} if the affine cone $C_X \subset \mbA^{n+1}$ of $X$, which is the closure of $\tau^{-1} (X)$ in $\mbA^{n+1}$, is smooth outside the origin $o$.
For a non-empty subset $S \subset \mbP$, we say that $X$ is {\it quasi-smooth along} $S$ if $C_X$ is smooth along $\overline{\tau^{-1} (S)} \subset \mbA^{n+1}$.

We say that $X$ is {\it well formed} if $\mbP$ is well formed and, for any $0 \le i < j \le n$ such that $\gcd \{a_0,\dots,\hat{a}_i,\dots,\hat{a}_j,\dots,a_n\} > 1$, $X$ does not contain the closed subset $(x_i = x_j = 0)$ of $\mbP$.
\end{Def} 

\begin{Rem} \label{rem:cl}
We note that for a quasi-smooth weighted complete intersection $X \subset \mbP (a_0,\dots,a_l)$ of dimension at least $3$, the Weil divisor class group $\Cl (X)$ is isomorphic to $\mbZ$ and is generated by a divisor (class) $A$ such that $\mcO_X (A) \cong \mcO_X (1)$.
Indeed, we have an exact sequence
\[
0 \to \mbZ \xrightarrow{\theta} \Cl (X) \to \Cl (R) \to 0,
\]
where $\theta (m) = m A$ and $R$ is the coordinate (graded) ring of the quasi-affine cone $C_X$ (see e.g.\ \cite[Theorem 1.6]{Watanabe}).
Now we have $\Cl (R) \cong \Cl (R_{\mfm})$, where $\mfm$ is the maximal ideal of the origin (see \cite[Corollary 10.3]{Fossum}). 
The latter is $0$ since $R_{\mfm}$ is a complete intersection local ring of dimension at least $4$ which is regular outside the maximal ideal (see \cite[Section 18]{Fossum}).
Thus $\Cl (X) = \mbZ \! \cdot \! A$.
\end{Rem}

In the rest of this section we assume that the ground field is an algebraically closed field $\K$ of positive characteristic $p$.

For a subset $I \subset \{0,1,\dots,n\}$, we define
\[
I_{\wt = 1} = \{\, i \in I \mid a_i = 1 \, \} \quad \text{and} \quad I_{\wt > 1} = I \setminus I_{\wt = 1}.
\]
We define
\[
\mbP^{\circ}_{\wt = 1} = \bigcup_{i \in \{0,\dots,n\}_{\wt = 1}} (x_i \ne 0),
\]
which is an open subset of $\mbP$.

\begin{Def}
For a non-empty subset $I = \{i_1,\dots,i_k\}$ of $\{0,\dots,n\}$, we define
\[
\Pi^*_{I,\mbP} = \left(\bigcap_{i \in I} (x_i \ne 0)\right) \cap \left(\bigcap_{j \notin I} (x_j = 0)\right) \subset \mbP,
\]
and call it the {\it coordinate stratum} of $\mbP$ with respect to $I$.
We denote by $\Pi_{I,\mbP}$ the closure of $\Pi^*_{I,\mbP}$.
\end{Def}

For $I = \{i_1,\dots,i_k\}$, we sometimes drop the subscript $\mbP$ and write $\Pi^*_{I,\mbP} = \Pi^*_I$, and also we write
\[
\Pi^*_{I,\mbP} = \Pi^*_{x_{i_1},\dots,x_{i_k}} \quad \text{and} \quad \Pi_{I,\mbP} = \Pi_{x_{i_1},\dots,x_{i_k}}.
\]

Let $I \subset \{0,\dots,n\}$ be a non-empty subset.
For a polynomial $h \in \K [x_0,\dots,x_n]$, we define $h|_{\Pi^*_I}$ to be the polynomial in variables $\{\,x_i \mid i \in I\,\}$ obtained by setting $x_j = 0$ in $h$ for all $j \in \{0,\dots,n\} \setminus I$.
For a matrix $M = (h_{ij})$ with entries $h_{ij} \in \K [x_0,\dots,x_n]$, we define $M|_{\Pi^*_I} = (h_{ij}|_{\Pi^*_I})$.

Let $\Lambda$ be a set of monomials in variables $x_0,\dots,x_n$.
For a ring $R$, we denote by $\langle \Lambda \rangle_R$ the free $R$-module generated by the monomials in $\Lambda$. 
In the following, we assume that $\Lambda$ is a set of monomials of the same weighted degree.
Then $\langle \Lambda \rangle_{\K} \subset H^0 (\mbP, \mcO_{\mbP} (d))$ is a $\K$-vector space.
We define $\mcL (\Lambda) \subset |\mcO_{\mbP} (d)|$ to be the linear system associated with $\langle \Lambda \rangle_{\K}$.
For subsets $\Xi = \{g_1,\dots,g_m\} \subset \Lambda$ and $J = \{l_1,\dots, l_k\} \subset \{0,\dots,n\}$, we define
\[
M_{\Xi, J} = \left( \frac{\prt \Xi}{\prt \{x_{l_1},\dots,x_{l_k}\}} \right) = \left( \frac{\prt g_j}{\prt x_{l_i}} \right)_{1 \le i \le k, 1 \le j \le m},
\]
and
\[
M'_{\Xi, J} = \left(\frac{\prt \Xi}{\prt \{x_{l_1},\dots,x_{l_k}\}}\right)' =
\begin{pmatrix}
g_1 & \cdots & g_m \\
& M_{\Xi, \{x_{l_1},\dots,x_{l_k}\}} &
\end{pmatrix}.
\]
We note that, while $M_{\Xi,J}$ is not defined when $J = \emptyset$, we define $M'_{\Xi,\emptyset} := (g_1 \ \cdots \ g_m)$.
We set $M_{\Lambda} = M_{\Lambda, \{0,\dots,n\}}$ and $M'_{\Lambda} = M'_{\Lambda,\{0,\dots,n\}}$. 
We will sometimes write $M_{\Xi, \{x_{l_1},\dots,x_{l_k}\}}$ and $M_{\Xi, \{x_{l_1},\dots,x_{l_k}\}}$ instead of $M_{\Xi,J}$ and $M'_{\Xi,J}$.

\subsection{A basic criterion}

Let $\mbP = \mbP (a_0,\dots,a_n)$ be a weighted projective space with homogeneous coordinates $x_0,\dots,x_n$ of weight $a_0,\dots,a_n$, respectively, and $\Lambda$ a set of monomials in $x_0,\dots,x_n$ of weighted degree $d$.

\begin{Lem} \label{lem:qsmcri1}
Let $I \subset \{0,\dots,n\}$ be a non-empty subset.
Suppose that $\rank M'_{\Lambda} (\msp) \ge |I|$ for any point $\msp \in \Pi^*_{I,\mbP}$.
Then a general member $X \in \mcL (\Lambda)$ is quasi-smooth along $\Pi^*_{I,\mbP}$.
\end{Lem}

\begin{proof}
We identify $W = \langle \Lambda \rangle_{\K}$ with $\K^{\lambda}$ via the basis $\Lambda$, where $\lambda = |\Lambda|$.
Then, for a point $\msp \in \Pi^*_I$, the kernel, denoted by $W_{\msp}$, of the map $W \cong \K^{\lambda} \to \K^{n+2}$ defined by the matrix $M'_{\Lambda} (\msp)$ is precisely the set of polynomials $f \in W$ such that $(f = 0) \in \mcL (\Lambda)$ is not quasi-smooth at $\msp$.
By the assumption $\rank M'_{\Lambda} (\msp) \ge |I|$, the codimension of $W_{\msp}$ in $W$ is at least $|I|$.
Then, by counting dimension keeping in mind that $\dim \Pi^*_{I,\mbP} = |I|-1$, we see that a general member of $\mcL (\Lambda)$ is quasi-smooth along $\Pi^*_{I,\mbP}$.
\end{proof}

\begin{Def}
For a non-empty subset $I \subset \{0,\dots,n\}$, we say that $\Lambda$ satisfies condition $(\ast)_{I,\mbP}$ (resp.\ $(\ast)'_{I,\mbP}$) if there are a subset $\Xi \subset \Lambda$ with $|\Xi| = |I|$ and a subset $J \subset \{0,\dots,n\}$ with $|J| = |I|$ (resp.\ $|J| = |I|-1$) such that 
\[
\det (M_{\Xi,J})|_{\Pi^*_{I,\mbP}} \quad (\text{resp.\ $\det (M'_{\Xi,J})|_{\Pi^*_{I,\mbP}}$}) 
\]
is a non-zero monomial.
We say that $\Lambda$ satisfies $(\dagger)_{I,\mbP}$ if it satisfies either $(*)_{I,\mbP}$ or $(*)'_{I,\mbP}$.
\end{Def}

It is clear from the above definition that if $\Lambda$ satisfies $(*)_{I,\mbP}$, $(*)'_{I,\mbP}$ or $(\dagger)_{I,\mbP}$, then $\Lambda'$ satisfies $(*)_{I,\mbP}$, $(*)'_{I,\mbP}$ or $(\dagger)_{I,\mbP}$, respectively, for any set $\Lambda'$ of monomials in $x_0,\dots,x_n$ containing $\Lambda$.

\begin{Lem} \label{lem:qsmcri2}
Let $I \subset \{0,\dots,n\}$ be a non-empty subset.
Suppose that $\Lambda$ satisfies $(\dagger)_{I,\mbP}$.
Then a general member of $\mcL (\Lambda)$ is quasi-smooth along $\Pi^*_{I,\mbP}$.
\end{Lem}

\begin{proof}
Put $\Pi^* = \Pi^*_{I,\mbP}$.
We see that $\det (M_{\Xi,J})|_{\Pi^*}$ and $\det (M'_{\Xi, J})|_{\Pi^*}$ are both $k \times k$ minors of $M'_{\Lambda}|_{\Pi^*}$.
The conditions $(*)_{I,\mbP}$ and $(*)'_{I,\mbP}$ imply that $\det (M_{\Xi,J}|_{\Pi^*}) (\msp) \ne 0$ and $\det (M'_{\Xi, J}|_{\Pi^*}) (\msp) \ne 0$ for any $\msp \in \Pi^*$, respectively.
Thus $\rank (M'_{\Lambda} (\msp)) \ge |I|$ for any $\msp \in \Pi^*$ and the assertion follows from Lemma \ref{lem:qsmcri1}.
\end{proof}

\subsection{Quasi-smoothness of special weighted hypersurfaces I}

Let $\mbP = \mbP (a_0,\dots,a_n)$ and $\tilde{\mbP} = \mbP (a_0,\dots,a_n,b)$ be weighted projective spaces with homogeneous coordinates $x_0,\dots,x_n$ and $x_0,\dots,x_n,w$ of weight $a_0,\dots,a_n$ and $a_0,\dots,a_n,b$, respectively.
Let $d$ be a positive integer divisible by $b$ and we set $m = d/b$.
We assume that $m$ is divisible by $p$.
The aim of this subsection is to make the quasi-smoothness criterion Lemma \ref{lem:qsmcri2} simpler for a general weighted hypersurface in $\tilde{\mbP}$ defined by an equation of the form $w^m + f (x_0,\dots,x_n) = 0$.
Let $\Lambda$ be a set of monomials of weighted degree $d = m b$ in variables $x_0,\dots,x_n$.
Note that $\Lambda$ does not contain a monomial involving $w$.
Note also that we think of $w$ as the $(n+1)$th coordinate $x_{n+1}$, so that, for example, we have
\[
\Pi^*_{\{0,\dots,n+1\}, \tilde{\mbP}} = \left(\bigcap_{i=0}^n (x_i \ne 0) \right) \cap (w \ne 0).
\]

\begin{Lem} \label{lem:qsmstdag}
Let $I \subset \{0,\dots,n\}$ be a non-empty subset.
If $\Lambda$ satisfies $(*)_{I,\mbP}$, then $\Lambda \cup \{w^m\}$ satisfies both $(\dagger)_{I,\tilde{\mbP}}$ and $(\dagger)_{I \cup \{n+1\},\tilde{\mbP}}$.
\end{Lem}

\begin{proof}
By the assumption, there are subsets $\Xi \subset \Lambda$ and $J \subset \{0,\dots,n\}$ such that $|\Xi| = |J| = |I|$ and $\det (M_{\Xi,J})|_{\Pi^*_{I,\mbP}}$ is a non-zero monomial.
It is obvious that $\Lambda \cup \{w^m\}$ satisfies $(*)_{I,\tilde{\mbP}}$ and hence $(\dagger)_{I,\tilde{\mbP}}$.
Since $m$ is divisible by $p$, we have
\[
\det \left(M'_{\Xi \cup \{n+1\}, J} \right)|_{\Pi^*_{I \cup \{n+1\}, \tilde{\mbP}}} 
= \pm w^m \det \left( M_{\Xi, J} \right)|_{\Pi^*_{I,\mbP}}.
\]
This shows that $\Lambda \cup \{w^m\}$ satisfies $(*)'_{I \cup \{n+1\},\tilde{\mbP}}$ and hence $(\dagger)_{I \cup \{n+1\},\tilde{\mbP}}$.
\end{proof}

\begin{Lem} \label{lem:qsmcritypeI}
Let $I \subset \{0,\dots,n\}$ be a non-empty subset.
If $\Lambda$ satisfies $(*)_{I',\mbP}$ for any non-empty subset $I' \subset I$, then the weighted hypersurface in $\tilde{\mbP}$ defined by $w^m + f = 0$ is quasi-smooth along $\Pi_{I \cup \{n+1\},\tilde{\mbP}}$ for a general $f \in \langle \Lambda \rangle_{\K}$.
\end{Lem}

\begin{proof}
We have
\[
\Pi_{I \cup \{n+1\},\tilde{\mbP}} = \left( \bigcup_{I' \subset I} \Pi^*_{I',\tilde{\mbP}} \right) \cup \left( \bigcup_{I' \subset I} \Pi^*_{I' \cup \{n+1\},\tilde{\mbP}} \right).
\]
It follows from Lemmas \ref{lem:qsmstdag} and \ref{lem:qsmcri2} that a general member of $\mcL (\Lambda \cup \{w^m\})$ is quasi-smooth along $\Pi_{I,\tilde{\mbP}}$, and the proof is completed. 
\end{proof}

The following gives an easy criterion for the condition $(*)_{I,\mbP}$ for $I \subset \{0,\dots,n\}$ with $|I| \le 2$.

\begin{Lem} \label{lem:qsmast}
Let $\Lambda$ be a set of monomials of degree $d$ in variables $x_0,\dots,x_n$.
\begin{enumerate}
\item For $i \in \{0,\dots,n\}$, $\Lambda$ satisfies $(*)_{{\{i\}}, \mbP}$ if and only if ether $x_i^k \in \Lambda$ for some $k$ with $p \nmid k$ or $x_i^l x_j$ for some $j \ne i$ and $l$.
\item For distinct $i_1, i_2 \in \{0,\dots,n\}$, $\Lambda$ satisfies $(*)_{\{i_1,i_2\},\mbP}$ if one of the following holds.
\begin{enumerate}
\item $x_{i_1}^{k_1}, x_{i_2}^{k_2} \in \Lambda$ for some $k_1,k_2$ such that $p \nmid k_1 k_2$.
\item $x_{i_1}^{k_1} x_{i_2}^{l_2}, x_{i_2}^{k_2} \in \Lambda$ for some $l, k_1, k_2$ such that $p \nmid k_1 k_2$.
\item $x_{i_1}^{k_1} x_{i_2}^{p l_2}, x_{i_1}^{l_1} x_{i_2}^{k_2} \in \Lambda$ for some $l_1,l_2,k_1,k_2$ such that $p \nmid k_1 k_2$.
\item $x_{i_1}^{l_1} x_{i_2}^{l_2} x_j, x_{i_1}^m x_{i_2}^{k} \in \Lambda$ for some $l_1, l_2, m, k$ such that $p \nmid k$ and $j \notin \{i_1,i_2\}$.
\item $x_{i_1}^{l_1} x_{i_2}^{l_2} x_{j_1}, x_{i_1}^{m_1} x_{i_2}^{m_2} x_{j_2} \in \Lambda$ for some $l_1, l_2, m_1, m_2$ and distinct $j_1, j_2 \notin \{i_1,i_2\}$.
\end{enumerate}
\end{enumerate}
\end{Lem}

\begin{proof}
It is easy to prove (1) and we leave it to readers.
We prove (2).
Let $\Xi$ be the set of the $2$ monomials given in (a), (b), (c), (d) or (e) and we set $\Pi^* = \Pi^*_{\{i_1,i_2\},\mbP}$.
Then we have
\[
\begin{split}
\det \left( M_{\Xi, \{x_{i_1},x_{i_2}\}} \right) |_{\Pi^*} &= k_1 k_2 x_{i_1}^{k_1-1} x_{i_2}^{k_2-1} \quad \text{(in case (a))}, \\
\det \left( M_{\Xi, \{x_{i_1},x_{i_2}\}} \right) |_{\Pi^*} &= k_1 k_2 x_{i_1}^{k_1-1} x_{i_2}^{l+k_2-1} \quad \text{(in case (b))}, \\
\det \left( M_{\Xi, \{x_{i_1},x_{i_2}\}} \right) |_{\Pi^*} &= k_1 k_2 x_{i_1}^{k_1+l_1-1} x_{i_2}^{k_2 + p l_2 -1}  \quad \text{(in case (c))}, \\
\det \left( M_{\Xi, \{x_j,x_{i_2}\}} \right) |_{\Pi^*} &= k x_{i_1}^{l_1+m} x_{i_2}^{l_2+k-1} \quad \text{(in case (d))}, \\
\det \left( M_{\Xi, \{x_{j_1},x_{j_2}\}} \right) |_{\Pi^*} &= x_{i_1}^{l_1+m_1} x_{i_2}^{l_2+m_2} \quad \text{(in case (e))}.
\end{split}
\]
Thus, any of the conditions (a), (b), (c) and (d) implies $(*)_{\{i_1,i_2\},\mbP}$.
\end{proof}

\subsection{Quasi-smoothness of special weighted hypersurfaces II}

In this subsection, let $\mbP = \mbP (a_0,\dots,a_n,c)$ be a weighted projective space with homogeneous coordinates $x_0,\dots,x_n$ and $v$ of weight $a_0,\dots,a_n$ and $c$.
Let $d$ be a positive integer such that $d = c m + a_k$ for some $k \in \{0,\dots,n\}$ and $m \ge 1$.
We fix such $k$.
The aim of this subsection is to make the quasi-smoothness criterion Lemma \ref{lem:qsmcri2} simpler for a general weighted hypersurface in $\mbP$ defined by an equation of the form $v^m x_k + f (x_0,\dots,x_n) = 0$.
Let $\Lambda$ be a set of monomials of weighted degree $d$ in variables $x_0,\dots,x_n$.

\begin{Def}
Let $I \subset \{0,\dots,n\}$ be a non-empty subset.
We say that $\Lambda$ satisfies $(\star)^k_{I,\mbP}$ if either there are  subsets $\Xi \subset \Lambda$ and $J \subset \{0,\dots,n\} \setminus \{k\}$ with $|\Xi| = |I|$ and $|J| = |I|$ such that 
\[
\det (M_{\Xi,J})|_{\Pi^*_{I,\mbP}}
\]
is a non-zero monomial or there are subsets $\Xi' \subset \Lambda$ and $J \subset \{0,\dots,n\} \setminus \{k\}$ with $|\Xi'| = |I|-1$ and $|J| = |I|$ such that 
\[
\det (M'_{\Xi',J})|_{\Pi^*_{I,\mbP}}
\] 
is a non-zero monomial.
\end{Def}

We drop the superscript $k$ from $(\star)^k_{I,\mbP}$ and denote it by $(\star)_{I,\mbP}$.
It is clear that if $\Lambda$ satisfies $(\star)_{I,\mbP}$, then it satisfies $(\dagger)_{I,\mbP}$.

\begin{Lem} \label{lem:qsmstar}
If $\Lambda$ satisfies $(\star)_{I,\mbP}$ for a non-empty subset $I \subset \{0,\dots,n\}$, then $\{v^m x_k\} \cup \Lambda$ satisfies both $(\dagger)_{I,\mbP}$ and $(\dagger)_{I \cup \{n+1\}, \mbP}$.
\end{Lem} 

\begin{proof}
It is clear that $\{v^m x_k\} \cup \Lambda$ satisfies $(\dagger)_{I,\mbP}$.
Suppose that there are subsets $\Xi \subset \Lambda$ and $J \subset \{0,\dots,n\} \setminus \{k\}$ such that $|\Xi| = |I|$, $|J| = |I|$ and $\det (M_{\Xi,J})|_{\Pi^*_{I,\mbP}}$ is a non-zero monomial.
Then we have 
\[
\det \left( M_{\{v x_k\} \cup \Xi,J \cup \{k\}} \right)|_{\tilde{\Pi}^*_{I \cup \{n+1\},\mbP}} = \pm v^m \det \left( M_{\Xi,J} \right)|_{\Pi^*_{I,\mbP}},
\] 
which shows that $\{v^m x_k\} \cup \Lambda$ satisfies $(\dagger)_{I \cup \{n+1\},\mbP}$.
Suppose that there are subsets $\Xi' \subset \Lambda$ and $J \subset \{0,\dots,n\} \setminus \{k\}$ such that $|\Xi'| = |I| - 1$, $|J| = |I|$ and $\det (M'_{\Xi',J})|_{\Pi^*_{I,\mbP}}$ is a non-zero monomial.
Then we have 
\[
\det \left( M'_{\{v x_k\} \cup \Xi',J \cup \{k\}} \right)|_{\Pi_{I \cup \{n+1\},\mbP}} = \pm v^m \det \left( M'_{\Xi,I} \right)|_{\Pi^*_{I,\mbP}},
\] 
which shows that $\{v^m x_k\} \cup \Lambda$ satisfies $(\dagger)_{I \cup \{n+1\},\mbP}$.
This completes the proof.
\end{proof}

The following gives a criterion for quasi-smoothness along $\mbP \setminus \mbP^{\circ}_{\wt > 1}$.

\begin{Lem} \label{lem:qsmtypeIIZinsep}
Suppose that $\Lambda$ satisfies $(\star)_{I,\mbP}$ for any non-empty subset $I \subset \{0,\dots,n\}_{\wt > 1}$.
Then the weighted hypersurface in $\mbP$ defined by $v^m x_k + f = 0$ is quasi-smooth along $\mbP \setminus \mbP^{\circ}_{\wt = 1}$ for a general $f \in \langle \Lambda \rangle_{\K}$.
\end{Lem}

\begin{proof}
We have
\[
\mbP \setminus \mbP^{\circ}_{\wt > 1} 
= \left( \bigcup_{I \subset \{0,\dots,n\}_{\wt > 1}} \Pi^*_{I,\mbP} \right)
\cup  \left( \bigcup_{I \subset \{0,\dots,n\}_{\wt > 1}} \Pi^*_{I \cup \{n+1\}.\mbP} \right)
\]
Thus the assertion follows from Lemmas \ref{lem:qsmstar} and \ref{lem:qsmcri2}.
\end{proof}

The following gives a criterion for quasi-smoothness when $m = 1$.

\begin{Lem} \label{lem:qsmtypeIIZ}
Suppose that $m = 1$ and $\Lambda$ satisfies $(\star)_{I,\mbP}$ for any non-empty subset $I \subset \{0,\dots,n\} \setminus \{k\}$.
Then the weighted hypersurface in $\mbP$ defined by $v x_k + f = 0$ is quasi-smooth for a general $f \in \langle \Lambda \rangle_{\K}$.
\end{Lem}

\begin{proof}
Let $f \in \langle \Lambda \rangle_{\K}$ be a general element and $X$ the hypersurface in $\mbP$ defined by $v x_k + f = 0$.
Since
\[
\frac{\prt (v x_k + f)}{\prt v} = x_k,
\]
we see that $X$ is quasi-smooth along the open set $(x_k \ne 0) \subset \mbP$.
We set
\[
\begin{split}
\mcI &= \{\, I \subset \{0,\dots,n\} \setminus \{k\} \mid I \ne \emptyset \,\}, \\ \mcI_{n+1} &= \{\, I' \subset \{0,\dots,n+1\} \setminus \{k\} \mid I' \ne \emptyset \,\}.
\end{split}
\]
By Lemma \ref{lem:qsmcri2}, it is enough to show that $\{v x_k\} \cup \Lambda$ satisfies $(\dagger)_{I,\mbP}$ for any $I \in \mcI_{n+1}$ since
\[
\mbP \supset (x_k = 0) = \bigcup_{I' \in \mcI_{n+1}} \Pi^*_{I',\mbP}.
\] 
This follows from the assumption and Lemma \ref{lem:qsmstar} since
\[
\mcI_{n+1} = \mcI \cup \{\, I \cup \{n+1\} \mid I \in \mcI \,\} \cup \{ \{n+1\}\}
\]
and $\{v x_k\} \cup \Lambda$ clearly satisfies $(\star)_{\{n+1\},\mbP}$.
\end{proof}

The following gives an easy criterion for the condition $(\star)_{I,\mbP}$ for $I \subset \{0,\dots,n\}$ with $|I| \le 2$.

\begin{Lem} \label{lem:checkstar}
Let $\Lambda$ be a set of monomials in variables $x_0,\dots,x_n$.
\begin{enumerate}
\item For $i \in \{0,\dots,n\}$, $\Lambda$ satisfies $(\star)^k_{\{i\},\mbP}$ if either $x_i^l \in \Lambda$ for some $l > 0$ or $x_i^l x_j$ for some $l > 0$ and $j \in \{0,\dots,n\} \setminus \{i,k\}$.
\item For distinct $i_1, i_2 \in \{0,\dots,n\} \setminus \{k\}$, $\Lambda$ satisfies $(\star)_{\{i_1,i_2\},\mbP}$ if one of the following holds.
\begin{enumerate}
\item $x_{i_1}^{l_1} x_{i_2}^{l_2}, x_{i_1}^{m_1} x_{i_2}^{m_2} \in \Lambda$ for some $l_1, l_2, m_1, m_2 \ge 0$ such that at least one of $l_1 - m_1$ and $l_2 - m_2$ is not divisible by $p$.
\item $x_{i_1}^{l_1}, x_{i_2}^{l_2} x_j \in \Lambda$ for some $l_1,l_2 > 0$ and $j \in \{0,\dots,n\} \setminus \{i_1,i_2,k\}$. 
\item $x_{i_1}^{l_1} x_{j_1}, x_{i_2}^{l_2} x_{j_2} \in \Lambda$ for some $l_1, l_2 > 0$ and distinct $j_1, j_2 \in \{0,\dots,n\} \setminus \{i_1,i_2,k\}$.
\end{enumerate}
\item For $i \in \{0,\dots,n\} \setminus \{k\}$, $\Lambda$ satisfies $(\star)_{\{i,k\},\mbP}$ if one of the following holds.
\begin{enumerate}
\item $x_k^{\alpha_1} x_i^{\beta_1} x_{j_1}, x_k^{\alpha_2} x_i^{\beta_2} x_{j_2} \in \Lambda$ for some $\alpha_1,\alpha_2,\beta_1,\beta_2 \ge 0$ and distinct $j_1,j_2 \in \{0,\dots,n\} \setminus \{k,i\}$.
\item $x_k^{\alpha} x_i^{\beta} x_j, x_k^{\delta} x_i^{\gamma} \in \Lambda$ for some $\alpha,\beta,\gamma,\delta \ge 0$ and $j \in \{0,\dots,n\} \setminus \{k,i\}$.
\item $x_k^{\alpha}, x_i^{\beta} x_k^{\gamma} \in \Lambda$ for some $\alpha,\beta > 0$ and $\gamma \ge 0$ such that $p \nmid \beta$.
\item $x_k^{\alpha} x_i, x_i^{\beta} \in \Lambda$ for some $\alpha,\beta > 0$ such that $p \nmid \beta -1$.
\end{enumerate}
\end{enumerate}
\end{Lem}

\begin{proof}
(1) follows easily since
\[
\det \left( M'_{\{x_i^l\},\emptyset} \right)|_{\Pi^*_{\{i\},\mbP}} = \det \left( M_{\{x_i^l x_j\}, \{x_j\}} \right) |_{\Pi^*_{\{i\},\mbP}} = x_i^l.
\]

We prove (2).
Set $\Pi^* = \Pi^*_{\{i_1,i_2\},\mbP}$.
Let $\Xi$ be the set of $2$ monomials indicated in (a), (b), (c) or (d).
Suppose that we are in case (a).
Then
\[
\begin{split}
\det \left(M'_{\Xi, \{x_{i_1}\}} \right)|_{\Pi^*} &= (m_1-l_1) x_{i_1}^{m_1+l_1-1} x_{i_2}^{m_2+l_2}, \\
\det \left(M'_{\Xi, \{x_{i_2}\}} \right)|_{\Pi^*} &= (m_2-l_2) x_{i_1}^{m_1+l_1} x_{i_2}^{m_2+l_2-1}.
\end{split}
\]
By the assumption, at least one of the above monomials is non-zero, and hence $\Lambda$ satisfies $(\star)_{\{i_1,i_2\},\mbP}$.
In the other cases, we have
\[
\begin{split}
\det \left(M'_{\Xi,\{x_j\}} \right) |_{\Pi^*} &= x_{i_1}^{l_1} x_{i_2}^{l_2} \quad \text{(in case (b))}, \\
\det \left(M_{\Xi, \{x_{j_1},x_{j_2}\}} \right) |_{\Pi^*} &= x_{i_1}^{l_1} x_{i_2}^{l_2} \quad \text{(in case (c))}.
\end{split}
\]
This shows that $\Lambda$ satisfies $(\star)_{\{i_1,i_2\},\mbP}$.

Finally we prove (3).
Set $\Pi^* = \Pi^*_{\{i,k\},\mbP}$ and let $\Xi$ be the set of $2$ monomials indicated in (a), (b), (c) or (d).
We have
\[
\begin{split}
\det \left(M_{\Xi,\{x_{j_1},x_{j_2}\}} \right) |_{\Pi^*} &= x_k^{\alpha_1+\alpha_2} x_i^{\alpha_2+\beta_2} \quad \text{(in case (a))}, \\
\det \left(M'_{\Xi, \{x_j\}} \right) |_{\Pi^*} &= x_k^{\alpha+\delta} x_i^{\beta+\gamma} \quad \text{(in case (b))}, \\
\det \left(M'_{\Xi, \{x_i\}} \right) |_{\Pi^*} &= \beta x_k^{\alpha+\gamma} x_i^{\beta-1} \quad \text{(in case (c))}, \\
\det \left(M'_{\Xi, \{x_i\}} \right) |_{\Pi^*} &= (\beta - 1) x_k^{\alpha} x_i^{\beta} \quad \text{(in case (d))}.
\end{split}
\]
Therefore $\Lambda$ satisfies $(\star)_{\{i,k\},\mbP}$.
\end{proof}

We make simpler the quasi-smoothness criterion given in Lemma \ref{lem:qsmtypeIIZ} when $n = 3$.
In the following lemma, we assume $k = 1$ for simplicity of the description.

\begin{Lem} \label{lem:criqsmZ}
Let $\mbP = \mbP (a_0,\dots,a_3,c)$ be a weighted hypersurface with homogeneous coordinates $x_0,\dots,x_3,v$ and let $\Lambda$ be a set of monomials in variables $x_0,\dots,x_3$ of degree $d$.
Suppose that one of the following holds.
\begin{enumerate}
\item $x_1^{l_1}, x_2^{l_2},x_3^{l_3} \in \Lambda$ for some $l_1,l_2,l_3$ such that at least two of them are not divisible by $p$.
\item $x_1^{l_1}, x_2^{l_2},x_3^{l_3} x_1 \in \Lambda$ for some $l_1,l_2,l_3$ such that $p \nmid l_3$ and either $p \nmid l_1$ or $p \nmid l_2$.
\item $x_1^{p l_1}, x_2^{l_2} x_1, x_3^{l_3} x_2 \in \Lambda$ for some $l_1,l_2,l_3$.
\item $x_1^{l_1} x_2, x_2^{l_2} x_3, x_3^{l_3} x_1 \in \Lambda$ for some $l_1,l_2,l_3$ such that $p \nmid l_1 l_2 l_3 + 1$.
\item $x_1^{l_1}, x_2 x_1^{l_2}, x_3 x_1^{l_3} \in \Lambda$ for some $l_1,l_2,l_3$ and $\Lambda$ satisfies $(\star)_{I,\mbP}$ for any non-empty subset $I \subset \{2,3\}$.
\end{enumerate}
Then the weighted hypersurface in $\mbP$ defined by $v x_0 + f = 0$ is quasi-smooth for a general $\langle \Lambda \rangle_{\K}$.
\end{Lem} 

\begin{proof}
By Lemma \ref{lem:qsmtypeIIZ}, it is enough to show that $\Lambda$ satisfies $(\star)_{I,\mbP}$ for any non-empty subset $I \subset \{1,2,3\}$.
By Lemma \ref{lem:checkstar}, it is easy to see that $\Lambda$ satisfies $(\star)_{I,\mbP}$ for any $I \subset \{1,2,3\}$ with $|I| \le 2$.
Thus it remains to show that $\Lambda$ satisfies $(\star)_{I,\mbP}$ for $I = \{1,2,3\}$.
In the following, we denote by $\Xi$ the set of $3$ monomials indicated in (1), (2), (3) or (4), and we set $\Pi^* = \Pi^*_{x_1,x_2,x_3}$.

Suppose that we are in case (1).
We may assume that $p \nmid l_1$ and $p \nmid l_2$.
We have
\[
\det \left( M'_{\Xi,\{x_1,x_2\}} \right)|_{\Pi^*} = l_1 l_2 x_1^{l_1-1} x_2^{l_2-1} x_3^{l_3},
\]
which verifies $(\star)_{I,\mbP}$.
Suppose that we are in case (2).
We have
\[
\det \left( M'_{\Xi, \{x_1,x_3\}} \right)|_{\Pi^*} = l_1 l_3 x_1^{l_1} x_2^{l_2} x_3^{l_3-1}, \ 
\det \left( M'_{\Xi, \{x_2,x_3\}} \right)|_{\Pi^*} = l_2 l_3 x_1^{l_1+1} x_2^{l_2-1} x_3^{l_3-1}.
\]
By the assumption, at least one of the above monomials is non-zero and thus $(\star)_{I,\mbP}$ is verified.

Suppose that we are in case (3).
We have
\[
\det \left( M'_{\Xi, \{x_1,x_2\}} \right)|_{\Pi^*} = x_1^{p l_1} x_2^{l_2} x_3^{l_3},
\]
which verifies $(\star)_{I,\mbP}$.

Suppose that we in case (4).
We have
\[
\det \left( M_{\Xi, \{x_1,x_2,x_3\}} \right)|_{\Pi^*} = (l_1 l_2 l_3 + 1) x_1^{l_1} x_2^{l_2} x_3^{l_3},
\]
which verifies $(\star)_{I,\mbP}$.

Finally suppose that we are in case (5).
It is clear that $\Lambda$ satisfies $(\star)_{\{1\},\mbP}$ since $x_1^{l_1} \in \Lambda$.
We have
\[
\det \left( M'_{\{x_1^{l_1},x_i x_1^{l_i}\}, \{x_i\}} \right)|_{\Pi^*_{x_1,x_i}} = x_1^{l_1 + l_i}
\]
for $i = 2,3$ and thus $\Lambda$ satisfies $(\star)_{\{1,i\},\mbP}$ for $i = 1,2$.
Further, we have
\[
\det \left( M'_{\{x_1^{l_1}, x_2 x_1^{l_2}, x_3 x_1^{l_3}\}, \{x_2,x_3\}} \right)|_{\Pi^*_{x_1,x_2,x_3}} = x_1^{l_1+l_2+l_3},
\]
which verifies $(\star)_{\{1,2,3\},\mbP}$.
We have verified $(\star)_{I,\mbP}$ for any $I \subset \{1,2,3\}$ and thus the proof is completed.
\end{proof}

\section{Critical points} \label{sec:critical}

The aim of this section is to show that a suitable section on a weighted projective space or a weighted hypersurface has only admissible critical points.
We introduce the following condition on positive integers $a_0,\dots,a_n$ and $d$.

\begin{Cond} \label{cdcrit}
\begin{enumerate}
\item The set $\{0,\dots,n\}_{\wt = 1}$ is non-empty, that is, there is $i \in \{0,\dots,n\}$ such that $a_i = 1$.
\item $d \ge 2 a_i$ for any $i \in \{0,\dots,n\}$.
\item If $p = 2$ and $n$ is odd, then there are distinct $j, k \in \{0,\dots,n\}$ such that $d \ge 3 a_j, 3 a_k$.
\end{enumerate}
\end{Cond}

\begin{Lem} \label{lem:crittypeI}
Let $\mbP = \mbP (a_0,\dots,a_n)$ be a weighted projective space.
Suppose that $a_0,\dots,a_n$ and $d$ satisfy \emph{Condition \ref{cdcrit}} and that $d$ is divisible by $p$.
Then, a general section $f \in H^0 (\mbP, \mcO_{\mbP} (d))$ has only admissible critical points on $\mbP^{\circ}_{\wt = 1}$. 
\end{Lem}

\begin{proof}
We denote by $x_0,\dots,x_n$ the homogeneous coordinates of $\mbP$ of weight $a_0,\dots,a_n$.
Let $\msp \in \mbP^{\circ}_{\wt = 1}$ be a point.
Then, replacing coordinates, we may assume $\msp = (1\!:\!0\!:\!\cdots\!:\!0)$.
Condition \ref{cdcrit}.(2) implies that the restriction map
\[
H^0 (\mbP, \mcO_{\mbP} (d)) \to \mcO_{\mbP} (d) \otimes (\mcO_{\mbP}/\mfm_{\msp}^3)
\]
is surjective.
If $p \ne 2$ or $p = 2$ and $n$ is even, then the assertion follows from \cite[18 Proposition]{Kollar1}.
Suppose that $p = 2$ and $n$ is odd.
Let $W_{\msp} \subset H^0 (\mbP,\mcO_{\mbP} (d))$ be the set of sections which have a critical point at $\msp$.
It is easy to see that $W_{\msp}$ is of codimension $3$ in $H^0 (\mbP,\mcO_{\mbP} (d))$.
We will construct a section $f$ which is contained in $W_{\msp}$ and has an admissible critical point at $\msp$.
Note that $a_0 = 1$ since we arrange coordinates so that $\msp = (1\!:\!0\!:\!\cdots\!:\!0) \in \mbP^{\circ}_{\wt = 1}$.
By Condition \ref{cdcrit}.(3), we may assume $d \ge 3 a_1$.
We define
\[
f = x_0^d + x_0^{d-2 a_1} x_1^2 + x_0^{d-a_2-a_3} x_2 x_3 + \cdots + x_0^{d-a_{n-1} - a_n} x_{n-1} x_n + x_0^{d-a_3} x_1^3 + \cdots,
\]
which is an element of $W_{\msp}$ and it has an admissible critical point at $\msp$.
Therefore the set of sections which has a non-admissible critical point at $\msp$ is of codimension at least $n+1$ and the assertion follows from the dimension counting argument.
\end{proof}

Next, let $\mbP := \mbP (a_0,\dots,a_n,c)$ be a weighted projective space with homogeneous coordinates $x_0,\dots,x_n$ and $v$ of weight $a_0,\dots,a_n$ and $c$, respectively.
We fix $k \in \{1,2,3\}$.
For a homogeneous polynomial $f = f (x_0,\dots,x_3)$ of weight $d := c + a_k$, we denote by $Z_f$ the hypersurface in $\mbP$ defined by $v x_k + f = 0$. 

\begin{Lem}
Suppose that $a_0,\dots,a_n$ and $d := c + a_k$ satisfy \emph{Condition \ref{cdcrit}} and that $c$ is divisible by $p$.
Then, for a general homogeneous polynomial $f = f (x_0,\dots,x_n)$ of weighted degree $d$, the section $v \in H^0 (Z_f, \mcO_{Z_f} (c))$ has only admissible critical points on $Z_f \cap \mbP^{\circ}_{\wt = 1}$.
\end{Lem}

\begin{proof}
We see that, on a point $\msp \in (x_k = 0) \cap Z_f$, $v$ (or its translation) can be chosen as a part of local coordinates, so that $v$ does not have a critical point at any point $\msp \in (x_k = 0) \cap Z_f$.

We set $U = (x_k \ne 0) \cap \mbP^{\circ}_{\wt = 1} \subset \mbP$. 
Let $\mcF$ be the affine space parametrizing the homogeneous polynomials in variables $x_0,\dots,x_3$ of weight $d$.
We define
\[
\begin{split}
\mcW^{\crit} &= \{\, (f, \msp) \mid \text{$v$ has a critical point at $\msp \in Z_f$}\, \} \subset \mcF \times U, \\
\mcW &= \{\, (f, \msp) \mid \text{$v$ has a non-admissible critical point at $\msp \in Z_f$}\, \} \subset \mcW.
\end{split}
\]
Let $\msp \in U$ be a point.
We will compute the dimension of the fibers $\mcW_{\msp}^{\crit}$ and $\mcW_{\msp}$ over $\msp$ of the projections $\mcW^{\crit} \to U$ and $\mcW \to U$, respectively.

By replacing coordinates other than $x_k$ and $v$, we may assume that  the coordinates other than $x_0,x_k,v$ vanish at $\msp$ and $a_0 = 1$.
We work on the open subset $U_0 = (x_0 \ne 0) \subset \mbP$ which we identify with the affine space with coordinates $x_1,\dots,x_n,v$.
 
Suppose that $k = 0$, that is, $Z_f$ is defined by $v x_0 + f = 0$.
Then, $Z_f \cap U_0$ is defined by $v + f(1,x_1,\dots,x_n) = 0$ and the point $\msp$ corresponds to $(0,\dots,0,\mu) \in U_0$ for some $\mu \in \K$.
We write $f (1,x_1,\dots,x_n) = \alpha + g_1 + g_2 + \cdots$, where $g_i = g_i (x_1,\dots,x_n)$ is homogeneous of degree $i$ (degree means the usual one; $\deg (x_i) = 1$).
Thus, $v = -f \in \mcW^{\crit}_{\msp}$ if and only if $\alpha = -\mu$ and $g_1 = 0$.
The latter imposes $n+1$ conditions.
If $p \ne 2$ or $p = 2$ and $n$ is even, then for an element $f \in \mcW^{\crit}_{\msp}$, we have $(f,\msp) \in \mcW_{\msp}$ if and only if the Hessian of $g_2$ is $0$, which imposes additional $1$ condition.
If $p = 2$ and $n$ is odd, then we can construct $f$ such that $g_1 = 0$, $g_2 = x_1^2 + x_2 x_3 + \cdots + x_{n-1} x_n$ and $g_3 = x_1^3$ since $d \ge 3 a_1$.
This shows $\mcW_{\msp} \ne \mcW^{\crit}_{\msp}$.
The above arguments show that the codimension of $\mcW_{\msp}$ in $\mcF \times U$ is at least $n+2$. 

Suppose that $k \ne 0$.
We may assume $k = 1$, that is, $Z_f$ is defined in $\mbP$ by $v x_1 + f = 0$.
Then, $Z_f \cap U_0$ is defined by $v x_1 + f (1,x_1,\dots,x_n) = 0$ and the point $\msp$ corresponds to $(\lambda,0,\dots,0,\mu)$ for some $\lambda,\mu \in \K$ with $\lambda \ne 0$.
We set $x^*_1 = x_1 - \lambda$.
We can write $f (1,x_1,\dots,x_n) = \alpha + g_1 + g_2 + \cdots$, where $g_i = g_i (x_1^*,x_2,\dots,x_n)$ is homogeneous of degree $i$.
Passing to the completion $\hat{\mcO}_{Z_f,\msp}$, we have
\[
\begin{split}
v &= - (x_1^* + \lambda)^{-1} (\alpha + g_1 + g_2 + g_3 + \cdots) \\
&= -(\lambda^{-1} - \lambda^{-2} x_1^* + \lambda^{-3} {x_1^*}^2 - \lambda^{-4} {x_1^*}^3 +  \cdots)(\alpha + g_1 + g_2 + g_3 + \cdots) \\
&= - \lambda^{-1} \alpha -  \lambda^{-1}(g_1 - \alpha \lambda^{-1} x_1^*) - \lambda^{-1} ( g_2 - \lambda^{-1} x_1^* g_1 + \alpha \lambda^{-2} {x_1^*}^2) \\
&\hspace{85pt} -\lambda^{-1} (g_3 - \lambda^{-1} x_1^* g_2 + \lambda^{-2} {x_1^*}^2 g_1 - \alpha \lambda^{-3} {x_1^*}^3) + \cdots
\end{split} 
\]
We see that $(f,\msp) \in \mcW^{\crit}_{\msp}$ if and only if $\lambda \mu + \alpha = 0$, $g_1 - \alpha \lambda^{-1} x_1^* = 0$.
The latter imposes $n+1$ condition since $d \ge 2 a_i$ for any $i$.
In case $(f,\msp) \in \mcW^{\crit}_{\msp}$, we have $\alpha = - \lambda \nu$ and $g_1 = -\nu x_1^*$, that is, $v = - \nu + \lambda^{-1} g_2 + \cdots$.
Thus, if $p \ne 2$ or $p = 2$ and $n$ is even, then we can conclude that $n+2$ conditions are imposed in order for $(f,\msp)$ to be contained $\mcW_{\msp}$.
If $p = 2$ and $n$ is odd, then we may assume $3 a_n \ge d$ and we can construct $f$ such that $\alpha = - \lambda \nu$, $g_1 = - \mu x_1^*$, $g_2 = x^*_1 x_2 + x_3 x_4 + \cdots + x_{n-2} x_{n-1} + x_n^2$ and $g_3 = x_n^3 + \cdots$. 
For such $f$, we see that $v$ has an admissible critical point at $\msp \in Z_f$.
This shows $\mcW_{\msp} \ne \mcW^{\crit}_{\msp}$ and thus the codimension of $\mcW_{\msp}$ in $\mcF \times U$ is at least $n+2$.
Therefore, by counting dimension, we conclude that $v$ has only admissible critical points on $Z_f \cap \mbP^{\circ}_{\wt = 1}$.
\end{proof}
 
\section{Orbifold Fano $3$-fold hypersurfaces: rationality} \label{sec:rat}

We say that a weighted hypersurface $X$ in $\mbP (a_0,\dots,a_{n+1})$ is a {\it weighted cone} if there exists $i \in \{0,\dots,n+1\}$ such that the defining equation of $X$ does not involve the variable $x_i$. 
The following rationality criterion is almost obvious.

\begin{Lem} \label{ratcri}
Let $a_0,a_1,\dots,a_{n+1}$ and $d$ be positive integers such that $a_0 \le a_1 \le \cdots \le a_{n+1}$.
Suppose that either $d < 2 a_{n+1}$ or $d = 2 a_{n+1} = 2 a_n$.
Then a weighted hypersurface of degree $d$ in $\mbP (a_0,a_1,\dots,a_{n+1})$ is rational if it is irreducible, reduced and is not a weighted cone.
\end{Lem}

\begin{proof}
Let $X$ be an irreducible and reduced weighted hypersurface of degree $d$ in $\mbP (a_0,a_1,\dots,a_{n+1})$ which is not a weighted cone.
Note that we have $d \ge a_{n+1}$ because otherwise $X$ is a weighted cone.

We claim that the defining equation of $X$ can be written in a form $x_{n+1} f + g = 0$ where $f,g$ are non-zero homogeneous polynomials in variables $x_0,\dots,x_n$.
If the assumption $d < 2 a_{n+1}$ is satisfied, then this is clear.
If the assumption $d = 2 a_n = 2 a_{n+1}$ is satisfied, then we may assume that $X$ passes through the point $\msp_{n+1} = (0\!:\!\cdots\!:\!0\!:\!1)$ after possibly changing homogeneous coordinates suitably.
With this choice of coordinates, It is clear that the defining equation of $X$ is in the desired form. 

Now it is easy to see that the projection $X \ratmap \mbP (a_0,a_1,\dots,a_n)$ from the point $\msp_{n+1}$ gives a birational map and $X$ is rational.
\end{proof}

\begin{Prop} \label{prop:rat}
A general member of $20$ families \emph{No.104--106}, \emph{111--115}, \emph{118--121}, and \emph{123--130} is rational.
\end{Prop}

\begin{proof}
This follows immediately from Lemma \ref{ratcri}.
\end{proof}

\section{Orbifold Fano $3$-fold hypersurfaces: stable non-rationality} \label{sec:stnonrat}

Among the $130$ families of orbifold Fano $3$-fold hypersurfaces, $20$ families are rational by Proposition \ref{prop:rat}, and stable non-rationality of a very general member of the 4 families No.~1, 3, 97, 98 have been known.
Furthermore, we do not treat cubic $3$-folds, the family No.~96.
The aim of this section is to prove stable non-rationality of very general members of the remaining $105$ families.
Although we do not treat the above mentioned families No.~1, 3, 97, 98, we remark that our argument can also be applied to those $4$ families.

We treat families No.~19, 103 and 122 separately in Sections \ref{sec:No19}, \ref{sec:No103} and \ref{sec:No122}, respectively.
The remaining $102$ families are divided into 2 groups named type I and type II, which consist of $65$ and $37$ families (see Tables \ref{table:typeI} and \ref{table:typeII}), and the proof of stable non-rationality will be given in Sections \ref{sec:typeI} and \ref{sec:typeII}, respectively.

\subsection{Type $\mathrm{I}$ families} \label{sec:typeI}

\begin{table}[h]
\begin{center}
\caption{Orbifold Fano 3-fold hypersurfaces of type $\mathrm{I}$}
\label{table:typeI}
\begin{tabular}{ccc|ccc}
\hline
No. & $X_d \subset \mbP (a_0,\dots,a_4)$ & $p$ & No. & $X_d \subset \mbP (a_0,\dots,a_4)$ & $p$ \\
\hline \quad \\[-4mm]
4 & $X_6 \subset \mbP (\underline{1},1,1,2,2)$ & $2$ & 64 & $X_{26} \subset \mbP (1,2,5,6,\underline{13})$ & $2$ \\[0.2mm]
5 & $X_7 \subset \mbP (\underline{1},1,1,2,3)$ & $7$ & 65 & $X_{27} \subset \mbP (1,2,5,\underline{9},11)$ & $3$ \\[0.2mm]
8 & $X_9 \subset \mbP (1,1,1,\underline{3},4)$ & $3$ & 67 & $X_{28} \subset \mbP (1,1,4,9,\underline{14})$ & $2$ \\[0.2mm]
10 & $X_{10} \subset \mbP (1,1,1,3,\underline{5})$ & $2$ & 68 & $X_{28} \subset \mbP (1,3,\underline{4},7,14)$ & $7$ \\[0.2mm]
11 & $X_{10} \subset \mbP (1,1,2,2,\underline{5})$ & $2$ & 70 & $X_{30} \subset \mbP (1,1,4,10,\underline{15})$ & $2$ \\[0.2mm]
13 & $X_{11} \subset \mbP (\underline{1},1,2,3,5)$ & $11$ & 71 & $X_{30} \subset \mbP (1,1,6,8,\underline{15})$ & $2$ \\[0.2mm]
14 & $X_{12} \subset \mbP (1,1,1,\underline{4},6)$ & $3$ & 73 & $X_{30} \subset \mbP (1,2,6,7,\underline{15})$ & $2$ \\[0.2mm]
15 & $X_{12} \subset \mbP (1,1,\underline{2},3,6)$ & $3$ & 74 & $X_{30} \subset \mbP (1,\underline{3},4,10,13)$ & $2$ \\[0.2mm]
17 & $X_{12} \subset \mbP (1,1,\underline{3},4,4)$ & $2$ & 76 & $X_{30} \subset \mbP (1,\underline{5},6,8,11)$ & $2$ \\[0.2mm]
20 & $X_{13} \subset \mbP (\underline{1},1,3,4,5)$ & $13$ & 79 & $X_{33} \subset \mbP (1,3,5,\underline{11},14)$ & $3$ \\[0.2mm]
21 & $X_{14} \subset \mbP (1,1,2,4,\underline{7})$ & $2$ & 80 & $X_{34} \subset \mbP (1,3,4,10,\underline{17})$ & $2$ \\[0.2mm]
22 & $X_{14} \subset \mbP (1,2,2,3,\underline{7})$ & $2$ & 81 & $X_{34} \subset \mbP (1,4,6,7,\underline{17})$ & $2$ \\[0.2mm] 
24 & $X_{15} \subset \mbP (1,1,2,\underline{5},7)$ & $3$ & 82 & $X_{36} \subset \mbP (1,1,5,12,\underline{18})$ & $2$ \\[0.2mm] 
25 & $X_{15} \subset \mbP (1,1,\underline{3},4,7)$ & $5$ & 84 & $X_{36} \subset \mbP (1,7,8,\underline{9},12)$ & $2$ \\[0.2mm]
26 & $X_{15} \subset \mbP (1,1,3,\underline{5},6)$ & $3$ & 85 & $X_{38} \subset \mbP (1,3,5,11,\underline{19})$ & $2$ \\[0.2mm]
27 & $X_{15} \subset \mbP (1,2,\underline{3},5,5)$ & $5$ & 86 & $X_{38} \subset \mbP (1,5,6,8,\underline{19})$ & $2$ \\[0.2mm]
28 & $X_{15} \subset \mbP (1,3,3,4,\underline{5})$ & $3$ & 87 & $X_{40} \subset \mbP (1,5,7,\underline{8},20)$ & $5$ \\[0.2mm]
34 & $X_{18} \subset \mbP (1,1,2,6,\underline{9})$ & $2$ & 88 & $X_{42} \subset \mbP (1,1,6,14,\underline{21})$ & $2$ \\[0.2mm]
36 & $X_{18} \subset \mbP (\underline{1},1,4,6,7)$ & $2$ & 89 & $X_{42} \subset \mbP (1,2,5,14,\underline{21})$ & $2$ \\[0.2mm]
41 & $X_{20} \subset \mbP (1,1,\underline{4},5,10)$ & $5$ & 91 & $X_{44} \subset \mbP (1,4,5,13,\underline{22})$ & $2$ \\[0.2mm]
45 & $X_{20} \subset \mbP (1,3,4,\underline{5},8)$ & $2$ & 92 & $X_{48} \subset \mbP (1,3,5,\underline{16},24)$ & $3$ \\[0.2mm]
46 & $X_{21} \subset \mbP (1,1,3,\underline{7},10)$ & $3$ & 93 & $X_{50} \subset \mbP (1,7,8,10,\underline{25})$ & $2$ \\[0.2mm]
47 & $X_{21} \subset \mbP (1,1,5,\underline{7},8)$ & $3$ & 94 & $X_{54} \subset \mbP (1,4,5,18,\underline{27})$ & $2$ \\[0.2mm]
48 & $X_{21} \subset \mbP (1,2,3,\underline{7},9)$ & $3$ & 95 & $X_{66} \subset \mbP (1,5,6,22,\underline{33})$ & $2$ \\[0.2mm]
49 & $X_{21} \subset \mbP (1,3,5,6,\underline{7})$ & $3$ & 99 & $X_{10} \subset \mbP (1,1,2,3,\underline{5})$ & $2$ \\[0.2mm]
50 & $X_{22} \subset \mbP (1,1,3,7,\underline{11})$ & $2$ & 101 & $X_{22} \subset \mbP (1,2,3,7,\underline{11})$ & $2$ \\[0.2mm]
51 & $X_{22} \subset \mbP (1,1,4,6,\underline{11})$ & $2$ & 102 & $X_{26} \subset \mbP (1,2,5,7,\underline{13})$ & $2$ \\[0.2mm]
52 & $X_{22} \subset \mbP (1,2,4,5,\underline{11})$ & $2$ & $107$ & $X_6 \subset \mbP (1,1,2,2,\underline{3})$ & $2$ \\[0.2mm]
53 & $X_{24} \subset \mbP (1,1,3,\underline{8},12)$ & $3$ & 109 & $X_{15} \subset \mbP (1,2,3,\underline{5},7)$ & $3$  \\[0.2mm]
59 & $X_{24} \subset \mbP (1,3,6,7,\underline{8})$ & $3$ & 110 & $X_{21} \subset \mbP (1,3,5,\underline{7},8)$ & $3$ \\[0.2mm]
61 & $X_{25} \subset \mbP (1,4,\underline{5},7,9)$ & $5$ & 116 & $X_{10} \subset \mbP (1,2,3,4,\underline{5})$ & $2$ \\[0.2mm]
62 & $X_{26} \subset \mbP (1,1,5,7,\underline{13})$ & $2$ & 117 & $X_{15} \subset \mbP (1,3,4,\underline{5},7)$ & $3$ \\[0.2mm]
63 & $X_{26} \subset \mbP (1,2,3,8,\underline{13})$ & $2$ & \\[0.2mm]
\end{tabular}
\end{center}
\end{table}

We consider families listed in Table \ref{table:typeI}.
The aim is to construct a subspace $T$ of the parameter space $\mbP^M$ of each family $\mcX \to \mbP^M$ satisfying Condition \ref{cd:T}.

We explain how to read Table \ref{table:typeI}.
In the 2nd and 5th columns, the weighted degree $d$ of the hypersurface and the ambient space $\mbP (a_0,\dots,a_4)$ is given.
Moreover, there is indicated a unique underlined weight.
We choose homogeneous coordinates $x,y,z,t,w$ of $\mbP (a_0,\dots,a_4)$ so that $w$ corresponds to the underlined weight and the others are ordered as $\wt (x) \le \wt (y) \le \wt (z) \le \wt (t)$.
For example, for family No.~4, $w,x,y,z,t$ are the coordinates of $\mbP (\underline{1},1,1,2,2)$ of weight $1,1,1,2,2$, respectively, and for family No.~8, $x,y,z,w,t$ are coordinates of $\mbP (1,1,1,\underline{3},4)$ of weights $1,1,1,3,4$, respectively.

In the following we treat type $\mathrm{I}$ families uniformly.
Let $\mcX \to \mbP^M$ be a type $\mathrm{I}$ family of weighted hypersurfaces of weighted degree $d$ in $\tilde{\mbP} = \mbP (a_0,\dots,a_4)$.
We assume that $a_4$ is the underlined weight and let $x_0,\dots,x_3$ and $w$ be the coordinates of $\mbP (a_0,\dots,a_4)$ of weight $a_0,\dots,a_3$ and $a_4$, respectively (When we treat a specific family individually, we use coordinates $x,y,z,t$ instead of $x_0,\dots,x_3$).
We work over an algebraically closed field $\K$ of characteristic $p$, where $p$ is the prime number given in the 3rd and 6th columns.
Let $\Lambda$ be the set of monomials in variables $x_0,\dots,x_3$ of weighted degree $d$.
We consider weighted hypersurfaces $X$ defined in $\tilde{\mbP}$ by an equation of the form
\[
w^m + f = 0,
\]
where $m = d/a_4$ is a positive integer and $f \in \langle \Lambda \rangle_{\K}$.
Those hypersurfaces are parametrized by $T_{\K}$, where $T \cong \mbA^N$ with $N = |\Lambda|$ is the parameter space of polynomials in $\langle \Lambda \rangle_{\mbZ}$.
Let $\pi \colon X \to \mbP = \mbP (a_0,\dots,a_4)$ be the projection which is the cyclic cover of $\mbP$ branched along the divisor $(f = 0) \subset \mbP$.
Note that the covering degree $m$ is divisible by $p$.
We set $\mbP^{\circ} = \mbP^{\circ}_{\wt = 1}$ and $X^{\circ} = \pi^{-1} (\mbP^{\circ})$.
In the following, we assume that $X$ is general, that is, $f$ is general in $\langle \Lambda \rangle_{\K}$.

\begin{Lem} \label{lem:typeIisolsing}
$X$ has only isolated cyclic quotient singularities along $X \setminus X^{\circ}$.
\end{Lem}

\begin{proof}
It is enough to show that $X$ is quasi-smooth along $X \setminus X^{\circ}$.
We set $I = \{0,1,2,3\}_{\wt > 1}$.
We have 
\[
X \setminus X^{\circ} = X \cap \left(\Pi^*_{I, \tilde{\mbP}} \cup \Pi^*_{I \cup \{4\},\tilde{\mbP}} \right).
\]
Hence, by Lemma \ref{lem:qsmcritypeI}, it is enough to show that $\Lambda$ satisfies $(*)_{I',\mbP}$ for any non-empty subset $I' \subset I$.

By Lemma \ref{lem:qsmast}, it is straightforward to check $(*)_{I',\mbP}$ for any $I' \subset I$ with $|I'| \le 2$ and we leave it to readers (see Examples \ref{ex:typeIqsm} below).
In particular, the proof is completed for families such that $|I| \le 2$.
In Table \ref{table:typeIqsm}, we list families (together with a set of monomials) such that $|I| \ge 3$.
For any such family, we have $|I| = 3$ and it remains to check $(*)_{I,\mbP}$ for $I = \{0,1,2,3\}_{\wt > 1}$.
Let $\Xi$ be the set of monomials in the 2nd or 4th column and let $J$ be the set of $3$ coordinates indicated as a subscript of $\Xi$.
Then it is straightforward to check that $\det (M_{\Xi,J}) |_{\Pi^*_{I,\mbP}}$ is a nonzero monomial, that is, $\Lambda$ satisfies $(*)_{I,\mbP}$.
This completes the proof.
\end{proof}

\begin{table}[h]
\begin{center}
\caption{Monomials proving quasi-smoothness along $X \setminus X^{\circ}$}
\label{table:typeIqsm}
\begin{tabular}{cc|cc}
\hline
No. & Monomials& No. & Monomials \\
\hline \quad \\[-4mm]
13 & $\{t^2 x,t z^2,z^3 y\}_{x,y,t}$ &
79 & $\{y^{11},t^2 z,z y^9 x\}_{x,y,z}$ \\[0.4mm]
20 & $\{t^2 y,z^2 t,z^3 x\}_{x,y,t}$ & 
80 & $\{t^3 z,z^6 t,y^{11} x\}_{x,z,t}$ \\[0.4mm]
22 & $\{y^7,z^7,t y^5 x\}_{x,y,z}$ & 
81 & $\{t^4 z,z^5 y,t^3 y^3 x\}_{x,y,z}$ \\[0.4mm]
27 & $\{y^7 x,z^3,t^3\}_{x,z,t}$ & 
84 & $\{t^3,y^4 z,y^5 x\}_{x,z,t}$ \\[0.4mm]
28 & $\{y^5,z^5,t^2 z^2 x\}_{x,y,z}$ & 
85 & $\{t^3 z,z^7 y,z^2 y^9 x\}_{x,y,z}$ \\[0.4mm]
36 & $\{z^3,t^2 y,t z y x\}_{x,y,z}$ & 
86 & $\{t^4 z,y^6 t,t^4 y x\}_{x,z,t}$ \\[0.4mm]
45 & $\{z^5,y^4 t,t^2 y x\}_{x,z,t}$ & 
87 & $\{t^2,y^8,z^2 y^5 x\}_{x,y,t}$ \\[0.4mm]
48 & $\{t^2 z,z^7,y^{10} x\}_{x,z,t}$ &
89 & $\{t^3,y^{21},z y^{18} x\}_{x,y,t}$ \\[0.4mm]
49 & $\{y^7,z^3 t,z^4 x\}_{x,y,t}$ &
91 & $\{y^{11},t^3 z, t z^6 x\}_{x,y,z}$ \\[0.4mm]
52 & $\{y^{11},t^2 z^3,t y^8 x\}_{x,y,z}$ &
92 & $\{t^2,z^9 y,x z y^{14}\}_{x,y,t}$ \\[0.4mm]
59 & $\{z^4,y^8,t^2 y^3 x\}_{x,y,z}$ &
93 & $\{t^5,y^6 z,y^7 x\}_{x,z,t}$ \\[0.4mm]
61 & $\{t^2 z,z^3 y,y^6 x\}_{x,y,t}$ &
94 & $\{t^3,z^{10} y,z y^{12} x\}_{x,y,t}$ \\[0.4mm]
63 & $\{y^{13},y^9 t,y^{11} z x\}_{x,y,t}$ &
95 & $\{t^3,z^{11},y^3 x\}_{x,z,t}$ \\[0.4mm]
64 & $\{y^{13},y^{10} t,z^5 x\}_{x,y,t}$ &
101 & $\{y^{11},z^5 t,t^3 x\}_{x,y,t}$  \\[0.4mm]
65 & $\{z^5 y,t^2 z,y^{13} x\}_{x,y,z}$ &
102 & $\{y^{13},t^2 z,z^5 x\}_{x,y,z}$  \\[0.4mm]
68 & $\{t^2,z^4,y^9 x\}_{x,z,t}$ &
109 & $\{z^5,y^4 t,y^7 x\}_{x,y,t}$ \\[0.4mm]
73 & $\{y^{15},z^5,t y^{11} x\}_{x,y,z}$ &
110 & $\{y^7,t^2 z,z^4 x\}_{x,y,z}$ \\[0.4mm]
74 & $\{z^3,z y^5,t z^4 x\}_{x,y,z}$ &
116 & $\{y^5,z^2 t,z^3 x\}_{x,y,t}$ \\[0.4mm]
76 & $\{y^5,t^2 z,t y^3 x\}_{x,y,z}$ &
117 & $\{y^5,z^2 t,t^2 x\}_{x,y,t}$
\end{tabular}
\end{center}
\end{table}

\begin{Ex} \label{ex:typeIqsm}
We consider family No.~22.
Let $X = X_{14} \subset \mbP (1,2,2,3,\underline{7})$ be a weighted hypersurface defined by $w^2 + f_{14} (x,y,z,t) = 0$, where $f_{14} \in \K [x,y,z,t]$ is general and $\K$ is of characteristic $2$. 
We set $\mbP = \mbP (1,2,2,3)$ and we have $\{0,1,2,3\}_{\wt > 1} = \{1,2,3\}$.
The existence of monomials $y^7, z^7,t^4 y \in \Lambda$ shows that $\Lambda$ satisfies $(*)_{I,\mbP}$ for any $I \subset \{1,2,3\}$ with $|I| = 1$.
For $I \subset \{1,2,3\}$ with $|I| = 2$, we have
\[
\left| \frac{\prt \{y^7,z^7\}}{\prt \{y,z\}} \right|_{\Pi^*_{y,z}} = y^6 z^6, \ 
\left| \frac{\prt \{y^7, t y^5 x\}}{\prt \{x,y\}} \right|_{\Pi^*_{y,t}} = t y^{11}, \ 
\left| \frac{\prt \{y^7, t z^5 x\}}{\prt \{x,z\}} \right|_{\Pi^*_{y,t}} = t z^{11}.
\]
Here (and after),
\[
\left| \frac{\prt \{y^7,z^7\}}{\prt \{y,z\}} \right|_{\Pi^*_{y,z}} = \det \left(M_{\{y^7,z^7\}, \{y,z\}}\right)|_{\Pi^*_{y,z}}
\]
and similarly for the others.
Finally, For $I = \{1,2,3\}$, we have
\[
\left| \frac{\prt \{y^7, z^7, t y^5 x\}}{\prt \{x,y,z\}} \right|_{\Pi^*_{y,z,t}} = t y^{11} z^6,
\]
The above computations show that $\Lambda$ satisfies $(*)_{I,\mbP}$ for any non-empty subset of $\{0,1,2,3\}_{\wt > 1}$ and thus $X$ is quasi-smooth along $X \setminus X^{\circ}$.
\end{Ex}

\begin{Lem} \label{lem:typeIcrit}
The section $f \in H^0 (\mbP, \mcO_{\mbP} (d))$ has only admissible critical points on $\mbP^{\circ}$.
\end{Lem}

\begin{proof}
It is straightforward to check that Condition \ref{cdcrit} is satisfied.
Thus, the assertion follows from Lemma \ref{lem:crittypeI}.
\end{proof}

\begin{Prop} \label{prop:typeIcheckT}
Any type $\mathrm{I}$ family $\mcX \to \mbP^M_{\mbZ}$ together with $T$ satisfies \emph{Condition \ref{cd:T}}.
\end{Prop}

\begin{proof}
We first check that Condition \ref{cdgen} is satisfied for $X$ and $Z$.
Note that in this case we have $Z = \mbP$ and $\bar{w} = f$.
It is clear that (1), (3) and (4) are satisfied.
By Lemmas \ref{lem:typeIisolsing} and \ref{lem:typeIcrit}, (3) is satisfied.

Thus, by Proposition \ref{prop:constr}, Condition \ref{cd:T}.(2) is satisfied.
Here the condition $T_{\K}^{\operatorname{indep}} \ne \emptyset$ follows if we choose $\K$ so that it is uncountable. 
Quasi-smoothness of general members of the subfamily $\mcX_{\mbC} \to \mbP^M_{\mbC}$ parametrized by $T_{\mbC}$ follows from quasi-smoothness criterion \cite[Theorem 3.3]{Okada2} in characteristic $0$.
Therefore Condition \ref{cd:T} is satisfied.
\end{proof}

\subsection{Type $\mathrm{II}$ families} \label{sec:typeII}

\begin{table}[h]
\begin{center}
\caption{Orbifold Fano 3-fold hypersurfaces of type $\mathrm{II}$}
\label{table:typeII}
\begin{tabular}{cccc|cccc}
\hline
No. & $X_d \subset \mbP (a_0,\dots,a_4)$ & $p$ & Eq & No. & $X_d \subset \mbP (a_0,\dots,a_4)$ & $p$ & Eq \\
\hline \quad \\[-4mm]
2 & $X_5 \subset \mbP (1,1,1,1,\underline{2})$ & $2$ & $w^2 x$ & 
43 & $X_{20} \subset \mbP (1,2,4,5,\underline{9})$ & $2$ & $w^2 y$ \\[0.2mm]
6 & $X_8 \subset \mbP (\underline{1},1,1,2,4)$ & $7$ & $w^7 x$ & 
44 & $X_{20} \subset \mbP (1,2,5,6,\underline{7})$ & $2$ & $w^2 t$ \\[0.2mm]
7 & $X_8 \subset \mbP (\underline{1},1,2,2,3)$ & $7$ & $w^7 x$ & 
54 & $X_{24} \subset \mbP (\underline{1},1,6,8,9)$ & $23$ & $w^{23} x$ \\[0.2mm]
9 & $X_9 \subset \mbP (1,1,\underline{2},3,3)$ & $2$ & $w^4 x$ & 
55 & $X_{24} \subset \mbP (1,2,3,\underline{7},12)$ & $3$ & $w^3 z$ \\[0.2mm]
12 & $X_{10} \subset \mbP (1,1,2,\underline{3},4)$ & $3$ & $w^3 x$ &
56 & $X_{24} \subset \mbP (1,2,3,8,\underline{11})$ & $2$ & $w^2 y$ \\[0.2mm]
16 & $X_{12} \subset \mbP (1,1,2,4,\underline{5})$ & $2$ & $w^2 z$ &
57 & $X_{24} \subset \mbP (1,3,4,\underline{5},12)$ & $2$ & $w^4 z$ \\[0.2mm]
18 & $X_{12} \subset \mbP (1,2,2,3,\underline{5})$ & $2$ & $w^2 z$ & 
58 & $X_{24} \subset \mbP (1,3,4,7,\underline{10})$ & $2$ & $w^2 z$ \\[0.2mm]
23 & $X_{14} \subset \mbP (1,2,\underline{3},4,5)$ & $3$ & $w^3 t$ & 
60 & $X_{24} \subset \mbP (1,4,5,6,\underline{9})$ & $2$ & $w^2 t$ \\[0.2mm]
29 & $X_{16} \subset \mbP (1,1,2,\underline{5},8)$ & $3$ & $w^3 x$ & 
66 & $X_{27} \subset \mbP (1,\underline{5},6,7,9)$ & $2$ & $w^4 z$ \\[0.2mm]
30 & $X_{16} \subset \mbP (1,1,\underline{3},4,8)$ & $5$ & $w^5 x$ & 
69 & $X_{28} \subset \mbP (1,4,6,7,\underline{11})$ & $2$ & $w^2 z$ \\[0.2mm]
31 & $X_{16} \subset \mbP (1,1,4,\underline{5},6)$ & $3$ & $w^3 x$ & 
72 & $X_{30} \subset \mbP (1,2,\underline{3},10,15)$ & $5$ & $w^5 t$ \\[0.2mm]
32 & $X_{16} \subset \mbP (1,2,\underline{3},4,7)$ & $5$ & $w^5 x$ & 
75 & $X_{30} \subset \mbP (1,\underline{4},5,6,15)$ & $3$ & $w^6 z$ \\[0.2mm]
33 & $X_{17} \subset \mbP (1,\underline{2},3,5,7)$ & $2$ & $w^8 x$ & 
77 & $X_{32} \subset \mbP (1,2,5,\underline{9},16)$ & $3$ & $w^3 z$ \\[0.2mm]
35 & $X_{18} \subset \mbP (\underline{1},1,3,5,9)$ & $17$ & $w^{17} x$ & 
78 & $X_{32} \subset \mbP (1,4,\underline{5},7,16)$ & $5$ & $w^5 z$ \\[0.2mm]
37 & $X_{18} \subset \mbP (1,2,\underline{3},4,9)$ & $3$ & $w^3 t$ & 
83 & $X_{36} \subset \mbP (1,3,4,\underline{11},18)$ & $3$ & $w^3 y$ \\[0.2mm]
38 & $X_{18} \subset \mbP (1,2,3,5,\underline{8})$ & $2$ & $w^2 y$ & 
90 & $X_{42} \subset \mbP (1,\underline{3},4,14,21)$ & $7$ & $w^7 t$ \\[0.2mm]
39 & $X_{18} \subset \mbP (1,3,4,\underline{5},6)$ & $3$ & $w^3 y$ & 
100 & $X_{18} \subset \mbP (1,2,3,\underline{5},9)$ & $3$ & $w^3 z$ \\[0.2mm]
40 & $X_{19} \subset \mbP (1,\underline{3},4,5,7)$ & $3$ & $w^6 x$ & 
108 & $X_{12} \subset \mbP (1,2,3,4,\underline{5})$ & $2$ & $w^2 y$ \\[0.2mm]
42 & $X_{20} \subset \mbP (1,2,\underline{3},5,10)$ & $5$ & $w^5 z$ &  
\end{tabular}
\end{center}
\end{table}

We consider type $\mathrm{II}$ families listed in Table \ref{table:typeII}.
In the 2nd and 6th columns, the weighted degree of the hypersurface and the ambient weighted projective space $\mbP (a_0,\dots,a_4)$ is given.
We choose homogeneous coordinates $x,y,z,t,w$ of $\mbP (a_0,\dots,a_4)$ so that $w$ corresponds to the underlined weight and the others are arranged as $\wt (x) \le \wt (y) \le \wt (z) \le \wt (t)$.

In the following we treat type $\mathrm{II}$ families uniformly.
Let $\mcX \to \mbP^M$ be a type $\mathrm{II}$ family of weighted hypersurfaces of weighted degree $d$ in $\tilde{\mbP} = \mbP (a_0,\dots,a_4)$.
We assume that $a_4$ is the underlined weight and let $x_0,\dots,x_3$ and $w$ be the coordinates of $\tilde{\mbP}$.
Let $\Lambda$ be the set of monomials in variables $x_0,\dots,x_n$ of weighted degree $d$.
We work over an algebraically closed field $\K$ of positive characteristic $p$, where $p$ is the prime number given in the 3rd or 7th column.
Let $\Lambda$ be the set of monomials in variables $x_0,\dots,x_3$ of weighted degree $d$.
We consider weighted hypersurfaces $X$ defined in $\tilde{\mbP}$ by an equation of the form
\[
w^m x_k + f = 0,
\]
where $w^m x_k$ is the monomials given in the 4th or 8th column and $f \in \langle \Lambda \rangle_{\K}$.
These hypersurfaces are parametrized by $T_{\K}$, where $T \cong \mbA^N_{\mbZ}$ with $N = |\Lambda|$ is the parameter space of polynomials in $\langle \Lambda \rangle_{\mbZ}$.
We define 
\[
Z = (\bar{w} x_k + f = 0) \subset \mbP = \mbP (a_0,\dots,a_3,m a_4),
\]
where $\bar{w}$ is the coordinate of weight $m a_4$, and let $\pi \colon X \to Z$ be the morphism defined as $\pi^*\bar{w} = w^m$.
Note that $m$ is divisible by $p$ and $\pi$ is an inseparable cyclic covering (of degree $m$) branched along the divisor $(\bar{w} = 0) \cap Z$.
We define $Z^{\circ} = Z \cap \mbP^{\circ}_{\wt = 1}$ and $X^{\circ} = \pi^{-1} (Z^{\circ})$.
In the following, we assume that $X$ is general.

\begin{Lem} \label{lem:nstrtypeIIqsmZ}
$Z$ is well formed and quasi-smooth.
In particular, $Z^{\circ}$ is nonsingular. 
\end{Lem}

\begin{proof}
It is straightforward to check that $Z$ is well formed, and we leave it to readers.
We prove quasi-smoothness of $Z$. 
In Table \ref{table:typeIIqsm1}, we list a set of monomials in the 2nd, 5th and 8th columns except for families No.~18, 23, 44, and this shows that $Z$ is quasi-smooth by applying $(j)$ of Lemma \ref{lem:criqsmZ}, where $(j)$ is the one given in the 3rd, 6th or 9th column. 

We consider family No.~18.
We have $x^{12},y x^{10},t x^9 \in \Lambda$ and thus, by Lemma \ref{lem:criqsmZ}.(5), it remains to check $(\star)$ for the strata $\Pi^*_{y,t}$, $\Pi^*_{y}$ and $\Pi^*_{t}$. 
We can check these easily by Lemma \ref{lem:checkstar} since $y^6, t^4,y^3 t^2 \in \Lambda$.

We consider families No.~23, 44 and 90, respectively.
We have $x^{14},y x^{12},zx^{10} \in \Lambda$, $x^{20},y x^{18},z x^{15} \in \Lambda$ and $x^{42}, y x^{38}, z x^{28} \in \Lambda$, respectively, and thus by Lemma \ref{lem:criqsmZ}.(5), it remains to check $(\star)$ for $\Pi^*_{y,z}, \Pi^*_y,\Pi^*_z$.
We can check these easily by Lemma \ref{lem:checkstar} since $y^7,z^3 y \in \Lambda$ for family No.~23, $y^{10},z^4,z^2 y^5 \in \Lambda$ for family No.~44 and $y^7 z, z^3 \in \Lambda$ for family No,~$90$.
This completes the proof.
\end{proof}

\begin{table}[h]
\begin{center}
\caption{Monomials proving quasi-smoothness of $Z$}
\label{table:typeIIqsm1}
\begin{tabular}{ccc|ccc|ccc}
\hline
No. & Monomials & \ref{lem:criqsmZ} & No. & Monomials & \ref{lem:criqsmZ} & No. & Monomials & \ref{lem:criqsmZ} \\
\hline \quad \\[-4mm]
2 & $\{y^5,z^5,t^5\}$ & (1) &
35 & $\{y^6,t^2,z^3 y\}$ & (2) &
60 & $\{y^6,z^4 y,x^{19} z\}$ & (3)  \\[0.4mm]
6 & $\{y^8,z^4,t^2\}$ & (1) &
37 & $\{y^9,z^4 y,x^{14} z\}$ & (3) &
66 & $\{t^3,x^{27},y^3 t\}$ & (2) \\[0.4mm]
7 & $\{z^4,y^4,t^2z\}$ & (2) &
38 & $\{z^6,t^3 z,x^{13} t\}$ & (3) &
69 & $\{t^4,y^7,x^{21} t\}$ & (2) \\[0.4mm]
9 & $\{y^9,z^3,t^3\}$ & (1) &
39 & $\{t^3,z^3 t,x^{14} z\}$ & (3) &
72 & $\{y^{15},z^3,x^{28} y\}$ & (2) \\[0.4mm]
12 & $\{z^5,y^{10},t^2 z\}$ & (2) &
40 & $\{t^2 z,z^3 y,y^3 t\}$ & (4) &
75 & $\{y^6,t^2,x^{25} y\}$ & (2) \\[0.4mm]
16 & $\{x^{12},t^3,y^{11} x\}$ & (2) &
42 & $\{y^{10},t^2,x^{18} y\}$ & (2) &
77 & $\{x^{32},y^{16},t^2\}$ & (1) \\[0.4mm]
18 & & (5) &
43 & $\{t^4,z^5,x^{15} t\}$ & (2) &
78 & $\{x^{32},y^8,t^2\}$ & (1) \\[0.4mm]
23 & & (5) &
44 & & (5) &
83 & $\{z^9,t^2,x^{32} z\}$ & (2) \\[0.4mm]
29 & $\{y^{16},z^8,t^2\}$ & (1) &
54 & $\{y^4,z^3,t^2 y\}$ & (2) &
90 & & (5) \\[0.4mm]
30 & $\{y^{16},z^4,t^2\}$ & (1) &
55 & $\{y^{12},t^2,x^{22} y\}$ & (2) &
100 & $\{y^9,t^2,x^{16} y\}$ & (2) \\[0.4mm]
31 & $\{z^4,y^{16},t^2 z\}$ & (2) &
56 & $\{z^8,t^3,x^{21} z\}$ & (2) &
108 & $\{z^4,t^3,x^9 z\}$ & (2) \\[0.4mm]
32 & $\{y^8,z^4,t^2 y\}$ & (2) &
57 & $\{t^2,y^4 t,x^{21} y\}$ & (3) &
& & \\[0.4mm]
33 & $\{t^2 y,z^2 t,y^4 z\}$ & (4) & 
58 & $\{y^6,t^3 y,x^{17} t\}$ & (3) &
& &
\end{tabular}
\end{center}
\end{table}

\begin{Lem}
$X$ has only isolated cyclic quotient singularities along $X \setminus X^{\circ}$.
\end{Lem}

\begin{proof}
We first claim that $X$ is quasi-smooth along $(x_k = 0) \subset \tilde{\mbP}$.
Let $\NQsm (X)$ and $\NQsm (Z)$ be the non-quasi-smooth loci of $X$ and $Z$, respectively.
We have $\prt (w^m x_k + f)/\prt w = 0$ since $p \mid m$, hence
\[
\NQsm (X) = \bigcap_{i=0}^3 \left( \frac{\prt (w^m x_k + f)}{\prt x_i} = 0 \right) \cap (w^m x_k + f = 0) \subset \tilde{\mbP}.
\]
By Lemma \ref{lem:nstrtypeIIqsmZ}, $Z$ is quasi-smooth, which implies
\[
\emptyset = \NQsm (Z) = \bigcap_{i=0}^3 \left(\frac{\prt (\bar{w} x_k + f)}{\prt x_i} = 0 \right) \cap (x_k = 0) \cap (\bar{w} x_k + f = 0) \subset \mbP.
\] 
We have
\[
\pi^{-1} (\NQsm (Z)) = \NQsm (X) \cap (x_k = 0),
\] 
and thus $\NQsm (X) \cap (x_k = 0) = \emptyset$, that is, $X$ is quasi-smooth along $(x_k = 0)$.

Let $X$ be a member of family for which $x_k$ is of weight $1$ (this corresponds to a family such that $w^m x$ is given in Table \ref{table:typeII}). 
In this case $X \setminus X^{\circ}$ is contained in $(x_k = 0)$ and thus $X$ is quasi-smooth along $X \setminus X^{\circ}$. 

We assume that the weight of $x_k$ is at least $2$ and we set $I = \{0,1,2,3\}_{\wt > 1}$.
By Lemma \ref{lem:qsmtypeIIZinsep}, it is enough to show that $(\star)^k_{I',\mbP}$ is satisfied for any non-empty subset $I' \subset I$.
By Lemma \ref{lem:checkstar}, it is straightforward to check $(\star)^k_{I',\mbP}$ for any subset $I' \subset I$ with $|I| \le 2$ and we leave it to readers.
In particular the proof is completed if $|I| \le 2$.
  
In Table \ref{table:typeIIqsm2}, we list families (together with a set of monomials) such that the weight of $x_k$ is at least $2$ and $|I| \ge 3$. 
For any such family, we have $|I| = 3$ and thus it remains to check $(\star)^k_{I,\mbP}$.
Let $\Xi$ be the set of $3$ monomials given in the 2nd, 4th or 6th column of the table and $J$ the set of $2$ coordinates given as the subscript of $\Xi$.
Then we see that $\det ( M'_{\Xi,J})|_{\Pi^*_{I,\mbP}}$ is a non-zero monomial and thus $\Lambda$ satisfies $(\star)_{I,\mbP}$, which completes the proof.
\end{proof}

\begin{table}[h]
\begin{center}
\caption{Monomials proving quasi-smoothness along $X \setminus X^{\circ}$}
\label{table:typeIIqsm2}
\begin{tabular}{cc|cc|cc}
\hline
No. & Monomials & No. & Monomials & No. & Monomials \\
\hline \quad \\[-4mm]
18 & $\{x t y^4, z^5 y, t^4\}_{x,y}$ & 
55 & $\{x z y^{10},t^2,y^{12}\}_{x,t}$ & 
75 & $\{x y z^4,t^2,y^6\}_{x,t}$ \\[0.4mm]
23 & $\{x t z^2,y^7,t^2 z\}_{x,y}$ & 
56 & $\{x z y^{10},t^3,z^8\}_{x,t}$ & 
77 & $\{x z y^{13},y^{16},t^2\}_{x,y}$  \\[0.4mm]
37 & $\{x t y^4,y^9,z^4 y\}_{x,y}$ & 
57 & $\{x y z^5,t y^4,y^8\}_{x,t}$ & 
78 & $\{x z y^6,y^8,t^2\}_{x,y}$ \\[0.4mm]
38 & $\{x z y^7,z^6,t^3 z\}_{x,z}$ & 
58 & $\{x t z^4, y^8, t^3 y\}_{x,y}$ & 
83 & $\{x y z^8,t^2,z^9\}_{x,t}$ \\[0.4mm]
39 & $\{x y^3 z^2,z^3 t,t^3\}_{x,t}$ & 
60 & $\{x z^3 y^2,z^4 y,y^6\}_{x,y}$ & 
90 & $\{x t y^5,z^3,z y^7\}_{x,z}$ \\[0.4mm]
42 & $\{x z y^7,y^{10},t^2\}_{x,t}$ & 
66 & $\{x z^2 y^2,t^3,t y^3\}_{x,y}$ & 
100 & $\{x t y^4,t^2,y^9\}_{x,t}$ \\[0.4mm]
43 & $\{x t y^7,z^5,t^4\}_{x,z}$ & 
69 & $\{x t y^5,y^7,t^4\}_{x,y}$ &
108 & $\{x z y^4,t^3,z^4\}_{x,t}$ \\[0.4mm]
44 & $\{x z y^7, z^2 y^5, z^4\}_{x,y}$ & 
72 & $\{x t y^7,z^3,y^{15}\}_{x,z}$
\end{tabular}
\end{center}
\end{table}

\begin{Prop} \label{prop:typeIIcheckT}
Any type $\mathrm{II}$ family $\mcX \to \mbP^M$ together with $T$ satisfies \emph{Condition \ref{cd:T}}.
\end{Prop}

\begin{proof}
We can verify Condition \ref{cd:T}.(1) by the quasi-smoothness criterion \cite[Theorem 3.3]{Okada2} in characteristic $0$.
We see that Condition \ref{cdgen} is satisfied by Lemmas \ref{lem:No103isolcyc} and \ref{lem:No103admcrit}, hence Condition \ref{cd:T}.(2) follows from Proposition \ref{prop:constr}.
\end{proof}

\subsection{Family No.~103} \label{sec:No103}

Let $\mcX \to \mbP^M$ be the family No.~103 consisting of the weighted hypersurfaces in $\mbP (2,3,5,11,19)$ of weighted degree $38$.
We set $\tilde{\mbP} = \mbP (2,3,5,19,11)$ and denote by $x,y,z,t,w$ the homogeneous coordinates of weight $2,3,5,11,19$, respectively.
We work over an algebraically closed field $\K$ of characteristic $2$.
Let $\Lambda$ be the set of monomials in variables $x,y,z,t$ of weighted degree $38$.
We consider weighted hypersurfaces $X \subset \tilde{\mbP}$ defined by an equation of the form
\[
w^2 + f (x,y,z,t) = 0.
\]
These $X$ are parametrized by $T_{\K}$, where $T \cong \mbA_{\mbZ}^N$ with $N = |\Lambda|$ parametrizes the polynomials in $\langle \Lambda \rangle_{\mbZ}$.
In the following we assume that $f$ is general.
We set $\mbP = \mbP (2,3,5,11)$ and $U = (x \ne 0) \cap (y \ne 0) \subset \mbP$.
Note that $U$ is smooth.
Let $\pi \colon X \to \mbP$ be the projection which is a double cover branched along the divisor $(f = 0) \subset \mbP$.

\begin{Lem} \label{lem:No103isolcyc}
$X$ is quasismooth along $X \setminus \pi^{-1} (U)$.
In particular $X$ has only isolated cyclic quotient singularities along $X \setminus \pi^{-1} (U)$.
\end{Lem}

\begin{proof}
We have
\[
X \setminus \pi^{-1} (U) = X \cap \left(\Pi_{I_0,\tilde{\mbP}} \cup \Pi_{I_1,\tilde{\mbP}} \right),
\]
where $I_0 = \{1,2,3\}$ and $I_1 = \{0,2,3\}$.
By Lemma \ref{lem:qsmcritypeI}, it is enough to show that $\Lambda$ satisfies $(*)_{I,\mbP}$ for any non-empty subset $I$ of either $I_0$ or $I_1$.
By Lemma \ref{lem:qsmast}, it is straight forward to check $(*)_{I,\mbP}$ for $I$ with $|I| \le 2$.
For the strata $\Pi^*_{I_0,\mbP} = \Pi^*_{y,z,t}$ and $\Pi^*_{I_1,\mbP} = \Pi^*_{x,z,t}$, we have
\[
\left| \frac{\prt \{t^3 z, z^7 y, y^{12} x\}}{\prt \{x, y, z\}} \right|_{\Pi^*_{y, z, t}} = t^3 z^7 y^{12}, \quad
\left| \frac{\prt \{t^3 z, z^7 y, x^{19}\}}{\prt \{x, y, z\}} \right|_{\Pi^*_{x, z, t}} = t^3 z^7 x^{18},
\]
which shows that $X$ is quasi-smooth along $X \setminus \pi^{-1} (U)$.
\end{proof}

\begin{Lem} \label{lem:No103admcrit}
The section $f \in H^0 (\mbP,\mcO_{\mbP} (38))$ has only admissible critical points on $U$.
\end{Lem}

\begin{proof}
We show that, for any point $\msp \in U$, the map 
\[
\rho \colon H^0 (\mbP, \mcO_{\mbP} (38)) \to \mcO_{\mbP} (38) \otimes \mcO_{\mbP}/ \mfm_{\msp}^2
\] 
is surjective and there exists $g \in H^0 (\mbP, \mcO_{\mbP} (38))$ which has an admissible critical point at $\msp$, which will complete the proof by the dimension counting argument.

We have an identification $U \cong (\mbA^1_u \setminus \{0\}) \times \mbA^2_{z,t}$.
where $z|_U$ and $t|_U$ are simply denoted by $z,t$, and $x|_U = y|_U = u$.
Replacing coordinates we may assume $\msp = (\lambda\!:\!1\!:\!0\!:\!0) \in U$ for some $\lambda \ne 0$ and, by the above identification, $\msp = (\lambda^2,0,0)$.
We set $u^* = u - \lambda^2$ so that $\mfm_{\msp} = (u^*,z,t)$. 
We have
\[
\rho (x^{10} y^6) = \lambda^{32}, \ 
\rho (x^{19}) = \lambda^{38} + \lambda^{38} u^*, \ 
\rho(z x^3 y^{9}) = \lambda^{24} z, \ 
\rho (t x^9 y^3) = \lambda^{24} t,
\] 
which implies that $\rho$ is surjective.
Moreover we have
\[
(z s x^8 y^2 + x^{19} + \lambda^2 x^{16} y^2 + \lambda^4 x^{13} y^4)|_U = \lambda^{38} + \lambda^{20} z s + \lambda^{32} {u^*}^{3} + \cdots,
\]
which has an admissible critical point at $\msp$.
This completes the proof.
\end{proof}

\begin{Prop} \label{prop:No103checkT}
The family \emph{No.~103} together with $T$ satisfies \emph{Condition \ref{cd:T}}.
\end{Prop}

\begin{proof}
We can verify Condition \ref{cd:T}.(1) by the quasi-smoothness criterion \cite[Theorem 3.3]{Okada2} in characteristic $0$.
We see that Condition \ref{cdgen} is satisfied by Lemmas \ref{lem:No103isolcyc} and \ref{lem:No103admcrit}, hence Condition \ref{cd:T}.(2) follows from Proposition \ref{prop:constr}.
\end{proof}

\subsection{Family No.~122} \label{sec:No122}

Let $\mcX \to \mbP^M$ be the family No.~122 consisting of the weighted hypersurfaces in $\tilde{\mbP} = \mbP (2,3,4,5,7)$ of weighted degree $14$.
We denote by $x,y,z,t,w$ the homogeneous coordinates of weight $2,3,4,5,7$, respectively.
We work over an algebraically closed field $\K$ of characteristic $2$.
Let $\Lambda$ be the set of monomials in variables $x,y,z,t$ of weighted degree $14$.
We consider weighted hypersurfaces $X \subset \tilde{\mbP}$ defined by an equation of the form
\[
w^2 + f(x,y,z,t) = 0.
\]
These $X$ are parametrized by $T_{\K}$, where $T \cong \mbA^N_{\mbZ}$ with $N = |\Lambda|$ parametrizes the polynomials in $\langle \Lambda \rangle_{\mbZ}$.
In the following we assume that $X$ is general.
We set $\mbP = \mbP (2,3,4,5)$ and $U = (x \ne 0) \cap (y \ne 0)$ of $\mbP (2,3,4,5)$.
Note that $U$ is smooth.
Let $\pi \colon X \to \mbP$ be the projection which is the double cover branched along the divisor $(f = 0) \subset \mbP$.
The arguments for this family are almost the same as in the previous subsection.

\begin{Lem} \label{lem:No122isolcyc}
$X$ is quasi-smooth along $X \setminus \pi^{-1} (U)$.
In particular $X$ has only isolated cyclic quotient singularities along $X \setminus \pi^{-1} (U)$.
\end{Lem}

\begin{proof}
As in the proof of Lemma \ref{lem:No103isolcyc}, it is enough to show that $\Lambda$ satisfies $(*)_{I,\mbP}$ for any non-empty subset $I$ of either $I_0$ or $I_1$, where $I_0 = \{1,2,3\}$ and $I_1 = \{0,2,3\}$.
By Lemma \ref{lem:qsmast}, it is easy to check $(*)_{I,\mbP}$ for $I$ with $|I| \le 2$.
For the strata $\Pi^*_{I_0,\tilde\mbP} = \Pi^*_{y,z,t}$ and $\Pi^*_{I_1,\mbP} = \Pi^*_{x,z,t}$, we have
\[
\left| \frac{\prt \{y^3 t, z^3 x, t^2 z\}}{\prt \{x, y, z\}} \right|_{\Pi^*_{y z t}} = t^3 z^3 y^2, \quad
\left| \frac{\prt \{x^7, t^2 z, t z y x\}}{\prt \{x, y, z\}} \right|_{\Pi^*_{x z t}} = t^3 z x^7,
\]
which show that $X$ is quasi-smooth along $X \setminus \pi^{-1} (U)$.
\end{proof}

\begin{Lem} \label{lem:No122admcrit}
The section $f \in H^0 (\mbP,\mcO_{\mbP} (14))$ has only admissible critical points on $U$.
\end{Lem}

\begin{proof}
We show that, for any point $\msp \in U$, the map 
\[
\rho \colon H^0 (\mbP, \mcO_{\mbP} (14)) \to \mcO_{\mbP} (14) \otimes \mcO_{\mbP}/ \mfm_{\msp}^2
\] 
is surjective and there exists $g \in H^0 (\mbP, \mcO_{\mbP} (14))$ which has an admissible critical point at $\msp$, which will complete the proof.
We have an identification $U \cong (\mbA^1_u \setminus \{0\}) \times \mbA^2_{z,s}$.
where $z|_U$ and $s|_U$ are simply denoted by $z,s$, and $x|_U = y|_U = u$.
Replacing coordinates we may assume $\msp = (\lambda\!:\!1\!:\!0\!:\!0) \in U$ for some $\lambda \ne 0$ and, by the above identification, $\msp = (\lambda^2,0,0)$.
We set $u^* = u - \lambda^2$ so that $\mfm_{\msp} = (u^*,z,s)$. 
We have
\[
\rho (x^4 y^2) = \lambda^{12}, \ 
\rho (x^7) = \lambda^{14} + \lambda^{12} u^*, \ 
\rho(z x^2 y^2) = \lambda^8 z, \ 
\rho (s x^3 y) = \lambda^8 s,
\] 
which implies that $\rho$ is surjective.
Moreover we have
\[
(z s x y + x^7 + \lambda^2 x^4 y^2 + \lambda^4 x y^4)|_U = \lambda^{14} + \lambda^4 z s + \lambda^8 {u^*}^{3} + \cdots,
\]
which has an admissible critical point at $\msp$.
This completes the proof.
\end{proof}

\begin{Prop} \label{prop:No122checkT}
The family \emph{No.~122} together with $T$ satisfies \emph{Condition \ref{cd:T}}.
\end{Prop}

\begin{proof}
We can verify Condition \ref{cd:T}.(1) by the quasi-smoothness criterion \cite[Theorem 3.3]{Okada2} in characteristic $0$.
We see that Condition \ref{cdgen} is satisfied by Lemmas \ref{lem:No122isolcyc} and \ref{lem:No122admcrit}, hence Condition \ref{cd:T}.(2) follows from Proposition \ref{prop:constr}.
\end{proof}

\subsection{Family No.~19} \label{sec:No19}

Let $\mcX \to \mbP^M$ be the family No.~19 consisting of the weighted hypersurfaces in $\mbP (1,2,3,3,4)$ of weighted degree $19$.
We re-order the weight and set $\tilde{\mbP} = \mbP (1,3,3,4,2)$ and denote by $x,y,z,t,w$ the homogeneous coordinates of weight $1,3,3,4,2$, respectively.
We work over an algebraically closed field $\K$ of characteristic $2$.
Let $\Lambda$ be the union of $\{w^4 t\}$ and the set of monomials in variables $x,y,z,t$ of weighted degree $12$.
We consider weighted hypersurfaces $X \subset \tilde{\mbP}$ defined by an equation of the form
\[
w^6 + \delta w^4 t + f (x,y,z,t) = 0,
\]
where $\delta \in \K$.
These $X$ are parametrized by $T_{\K}$, where $T \cong \mbA^N_{\mbZ}$ is the space parametrizing the polynomials in $\langle \Lambda \rangle_{\mbZ}$.

In the following, we assume that $X$ is general.
By re-scaling $t$, we assume that $\delta = 1$.
We define
\[
Z = (\bar{w}^3 + \bar{w}^2 t + f_{12} = 0) \subset \mbP := \mbP (1,3,3,4,4),
\]
where $\bar{w}$ is the coordinate of weight $4$ other than $t$ (so that $x,y,z,t,\bar{w}$ are the coordinates of $\mbP$), and let $\pi \colon X \to Z$ be the morphism defined by $\pi^* \bar{w} = w^2$, which is a double cover of $Z$ branched along the divisor $(\bar{w} = 0) \cap Z$.

\begin{Lem} \label{lem:No19qsmZ}
$Z$ is quasi-smooth.
\end{Lem}

\begin{proof}
We see that $Z$ is quasi-smooth along $(\bar{w} \ne 0)$ since
\[
\frac{\prt (\bar{w}^3 + \bar{w}^2 t + f_{12})}{\prt \bar{w}} = \bar{w}^2.
\]
Let $\overline{\Lambda}$ be the union of $\{\bar{w}^3, \bar{w}^2 t\}$ and the set of monomials in variables $x,y,z,t$ of weighted degree $12$, so that $Z$ is a general member of $\mcL (\overline{\Lambda})$.
By Lemma \ref{lem:qsmcri2}, it is enough to show that $\overline{\Lambda}$ satisfies $(\dagger)_{I,\mbP}$ for any $I$ such that $\Pi_{I,\mbP}^* \subset (\bar{w} = 0)$.
The existence of monomials $x^{12},y^4,z^4,t^3 \in \overline{\Lambda}$ shows that $(\dagger)$ is satisfied for each vertex, i.e., $\Pi^*_x, \Pi^*_y,\Pi^*_z, \Pi^*_t$.
For the $1$-dimensional strata, we have
\[
\left| \frac{\prt \{x^9 z,y^4\}}{\prt z} \right|'_{\Pi^*_{x,y}} = x^9 y^4, \quad
\left| \frac{\prt \{x^9 y,z^4\}}{\prt z} \right|'_{\Pi^*_{x,z}} = x^9 z^4, \quad
\left| \frac{\prt \{x^{12},t^3\}}{\prt t} \right|'_{\Pi^*_{x,t}} = x^{12} t^2,
\]
\[
\left| \frac{\prt \{y^4,z^3 y\}}{\prt y} \right|'_{\Pi^*_{y,z}} = y^4 z^3, \quad
\left| \frac{\prt \{y^4,t^3\}}{\prt t} \right|'_{\Pi^*_{y,t}} = y^4 t^2, \quad
\left| \frac{\prt \{z^4,t^3\}}{\prt t} \right|'_{\Pi^*_{z,t}} = z^4 t^2.
\]
Here, for example,
\[
\left| \frac{\prt \{x^9 z,y^4\}}{\prt z} \right|'_{\Pi^*_{x,y}} = \det \left(M'_{\{x^9 z,y^4\},\{z\}} \right)|_{\Pi^*_{x,y}}
\]
and similarly for the others.
For $2$-dimensional strata, we have
\[
\left| \frac{\prt \{z^3 y,z x^9,x^{12}\}}{\prt \{y,z\}} \right|'_{\Pi^*_{x,y,z}} = z^3 x^{21}, \quad
\left| \frac{\prt \{t^3,y x^9,x^{12}\}}{\prt \{y,t\}} \right|'_{\Pi^*_{x,y,t}} = t^2 x^{21},
\]
\[
\left| \frac{\prt \{t^3,z x^9,x^{12}\}}{\prt \{z,t\}} \right|'_{\Pi^*_{x,z,t}} = t^2 x^{21}, \quad
\left| \frac{\prt \{t^3,y z^3,y^4\}}{\prt \{y,t\}} \right|'_{\Pi^*_{y,z,t}} = t^2 z^3 y^4.
\]
Finally, for the $3$-dimensional stratum $\Pi^*_{x,y,z,t}$, we have
\[
\left| \frac{\prt \{t^3,z^3 y,z x^9, x^{12}\}}{\prt \{y,z,t\}} \right|'_{\Pi^*_{x,y,z,t}} = t^2 z^3 x^{21}.
\]
Thus $Z$ is quasi-smooth.
\end{proof}

We set $Z^{\circ} = Z \cap \mbP^{\circ}_{\wt = 1} = Z \cap (x \ne 0)$ and $X^{\circ} = \pi^{-1} (Z^{\circ}) = X \cap (x \ne 0)$.

\begin{Lem} \label{lem:No19cyccing}
$X$ has only isolated cyclic quotient singularities along $X \setminus X^{\circ}$.
\end{Lem}

\begin{proof}
We claim that $X$ is quasi-smooth along $X \cap (w = 0)$.
We denote by $\NQsm (X)$ and $\NQsm (Z)$ the non-quasi-smooth loci of $X$ and $Z$, respectively.
Then it is easy to check that 
\[
\NQsm (X) \cap (w = 0) = \pi^{-1} (\NQsm (Z)),
\]
which proves our claim since $\NQsm (Z) = \emptyset$ by Lemma \ref{lem:No19qsmZ}.

Note that $X \setminus X^{\circ} = X \cap (x = 0)$.
We need to show that $X$ is quasi-smooth along $X \cap (x = 0) \cap (w \ne 0)$.
By Lemma \ref{lem:qsmcri2}, it is enough to show that $\Lambda \cup \{w^6\}$ satisfies the condition $(\dagger)_{I,\tilde{\mbP}}$ for any $I$ such that $x$ vanishes along $\Pi^*_{I,\tilde{\mbP}}$ but $w$ does not.  
Specifically, it is enogh to check $(\dagger)$ for the strata 
\[
\Pi^*_w, \ \Pi^*_{y,w}, \ \Pi^*_{z,w}, \ \Pi^*_{t,w}, \ \Pi^*_{y,z,w}, \ \Pi^*_{y,t,w}, \ \Pi^*_{z,t,w}, \ \Pi^*_{y,z,t,w}.
\]
It is clear that $(\dagger)$ is satisfied for the $0$-dimensional stratum $\Pi^*_w$ since $w^6 \in \Lambda \cup \{w^6\}$.
For $1$-dimensional strata, we have
\[
\left| \frac{\prt \{w^6,y^3 z\}}{\prt z} \right|'_{\Pi^*_{y,w}} = w^6 y^3, \quad
\left| \frac{\prt \{w^6,z^3 y\}}{\prt y} \right|'_{\Pi^*_{z,w}} = w^6 z^3, \quad
\left| \frac{\prt \{w^6,t^3\}}{\prt t} \right|'_{\Pi^*_{t,w}} = w^6 t^2.
\]
For $2$-dimensional strata, we have
\[
\left| \frac{\prt \{w^4 t,y^4,y^3 z\}}{\prt \{z,t\}} \right|'_{\Pi^*_{y,z,w}} = w^4 y^7,
\]
and
\[
\left| \frac{\prt \{w^6,t^3,z y^3\}}{\prt \{z,t\}} \right|'_{\Pi^*_{y,t,w}} = w^6 t^2 y^3, \quad 
\left| \frac{\prt \{w^6,t^3,z^3 y\}}{\prt \{y,t\}} \right|'_{\Pi^*_{z,t,w}} = w^6 t^2 z^3
\]
Finally, for the $3$-dimensional stratum $\Pi^*_{y,z,t,w}$, we have
\[
\left| \frac{\prt \{w^6,t^3,z^3 y,t^2 z x\}}{\prt \{x,y,t\}} \right|'_{\Pi^*_{y,z,t,w}} = w^6 t^4 z^4.
\]
Therefore $X$ is quasismooth along $X \setminus X^{\circ}$.
\end{proof}

\begin{Lem} \label{lem:No19admcrit}
The section $\bar{w} \in H^0 (Z, \mcO_Z (4))$ has only admissible critical points on $Z^{\circ}$.
\end{Lem}

\begin{proof}
For $g \in \langle \Lambda \rangle_{\K}$, let $Z_g$ be the weighted hypersurface in $\mbP$ defined by $\bar{w}^3 + \bar{w}^2 t + g = 0$ (so that we have  $Z = Z_f$).
We see that, for a point $\msp \in (\bar{w} = 0) \cap Z_g$, the section $\bar{w}$ can be chosen as a part of local coordinates of $Z_g$ at $\msp$, so that $\bar{w}$ does not have a critical point at any point of $\msp \in Z_g \cap (\bar{w} = 0)$.

We set $U = (\bar{w} \ne 0) \cap \mbP^{\circ}_{\wt = 1} \subset \mbP$.
Let $\mcF$ be the affine space parametrizing the homogeneous polynomials in variables $x,y,z,t$ of weighted degree $19$ and define
\[
\begin{split}
\mcW^{\crit} &= \{\, (g,\msp) \mid \text{$\bar{w}$ has a critical point at $\msp \in Z_g$} \,\} \subset \mcF \times U, \\
\mcW^{\operatorname{na}} &= \{\, (g,\msp) \mid \text{$\bar{w}$ has a non-admissible critical point at $\msp \in Z_g$} \,\} \subset \mcW^{\crit}.
\end{split} 
\]
Let $\msp \in U$ be a point.
We will compute the dimension of the fibers $\mcW_{\msp}^{\crit}$ and $\mcW_{\msp}^{\operatorname{na}}$ over $\msp$ of the projections $\mcW^{\crit} \to U$ and $\mcW^{\operatorname{na}} \to U$, respectively.
Since $\msp \in U \subset \mbP^{\circ}_{\wt = 1}$ and $\mbP^{\circ}_{\wt = 1} = (x \ne 0)$, the section $x$ does not vanish at $\msp$ and thus we may assume $\msp = (1\!:\!0\!:\!0\!:\!\lambda\!:\!\mu)$ for some $\lambda,\mu \in \K$ with $\mu \ne 0$ after replacing $y$ and $z$.

By setting $x = 1$, we think of $U$ as an open subset $(\bar{w} \ne 0)$ of the affine space $\mbA^4$ with coordinates $y,z,t,\bar{w}$.
The point $\msp$ corresponds to $(0,0,\lambda,\mu) \in \mbA^4$.
We see that $Z_g \cap U$ is defined by $\bar{w}^3 + \bar{w}^2 t + g (1,y,z,t) = 0$.
We set $t^* = t - \lambda$ and write $g (1,y,z,t) = \alpha + g_1 + g_2 + \cdots$, where $g_i = g_i (y,z,t^*)$ is homogeneous of degree $i$.
Note that $y,z,t^*$ can be chosen as local coordinates of $Z_g$ at $\msp$.
Passing to the completion $\hat{\mcO}_{Z_g,\msp} \cong \K [[y,z,t^*]]$, we think of $\bar{w} = \bar{w} (y,z,t^*)$ as a formal power series in $y,z,t^*$.
We write $\bar{w} = \mu + h_1 + h_2 + \cdots$, where $h_i = h_i (y,z,t^*)$ is homogeneous of degree $i$.
By looking at the constant terms and the degree $1$ terms in the equation $\bar{w}^3 + \bar{w}^2 t + \alpha + g_1 + g_2 + \cdots = 0$, we have the relations:
\[
\alpha = \mu^3 + \mu^2 \lambda, \ 
h_1 = t^* + \lambda^{-2} g_1.
\]
Note that $\alpha = \mu^3 + \mu^2 \lambda$ is equivalent to the condition $\msp \in Z_g$.
The section $\bar{w}$ has a critical point if and only if $h_1 = 0$, that is, $t^* + \mu^{-2} g_1 = 0$.
This implies that $4$ conditions are imposed in order for $(g,\msp)$ to be contained in $\mcW^{\crit}_{\msp}$, that is, $\mcW^{\crit}_{\msp}$ is of codimension $4$ in $\mcF \times \{\msp\}$.

In the following, we keep the above setting and we assume that $(g,\msp) \in \mcW^{\crit}_{\msp}$ and we will show that $\mcW^{\operatorname{na}}_{\msp} \ne \mcW^{\crit}_{\msp}$.
Now we have $\alpha = \mu^3 + \mu^2 \lambda$ and $h_1 = t^* + \mu^{-2} g_1 = 0$.
By looking at the degree $2$ and $3$ terms in the defining equation of $Z_g \cap U$, we have
\[
h_2 = \mu^{-2} g_2, \ 
h_3 = \mu^{-2} g_3,
\]
that is,
\[
\bar{w} = \mu + \mu^{-2} g_2 + \mu^{-3} g_3 + \cdots.
\]
We explicitly construct $g$ as follows:
\[
g = \mu x^{19} + \mu^2 (t-\lambda x^4) x^{15} + \mu^2 (yz x^{13} + \mu^2 (t-\lambda x^4)^2 x^{11}) + (t-\lambda x^4)^3 x^7 + \cdots,
\]
For the above $g$, we have $g_1 = \mu^2 t^*$, $g_2 = \mu^2 (y z + {t^*}^2)$ and $g_3 = \mu^2 ({t^*}^3 + \cdots)$, so that $\bar{w}$ has an admissible critical point at $\msp$.
This shows that $\mcW^{\operatorname{na}}_{\msp} \ne \mcW^{\crit}_{\msp}$. 
Therefore, since $\dim U = 4$, we conclude the section $\bar{w}$ has only admissible critical point on $Z^{\circ} = Z \cap U$ for a general $f$.
\end{proof}

\begin{Prop} \label{prop:No19checkT}
The family \emph{No.~19} together with $T$ satisfies \emph{Condition \ref{cd:T}}.
\end{Prop}

\begin{proof}
We can verify Condition \ref{cd:T}.(1) by the quasi-smoothness criterion \cite[Theorem 3.3]{Okada2} in characteristic $0$.
We see that Condition \ref{cdgen} is satisfied by Lemmas \ref{lem:No19qsmZ}, \ref{lem:No19cyccing} and \ref{lem:No19admcrit}, hence Condition \ref{cd:T}.(2) follows from Proposition \ref{prop:constr}.
\end{proof}

Now Theorems \ref{mainthm1} and \ref{mainthm2} follow from Propositions \ref{prop:rat}, \ref{prop:typeIcheckT}, \ref{prop:typeIIcheckT}, \ref{prop:No103checkT}, \ref{prop:No122checkT} and \ref{prop:No19checkT}.

\section{Example of non-rational Fano $3$-folds and absolute complexity} \label{sec:ex}

We recall the rationality criterion in terms of absolute complexity given in \cite{BMSZ}.

\begin{Def}[{\cite[Definition 1.7]{BMSZ}}]
Let $X$ be a proper variety of dimension $n$ and let $(X,\Delta)$ be a log pair. 
The {\it absolute complexity} $\gamma = \gamma (X,\Delta)$ of $(X,\Delta)$ is $n+\rho-d$, where $\rho$ is the rank of the group of Weil divisors modulo algebraic equivalence and $d$ is the sum of the coefficients of $\Delta$.
\end{Def}

\begin{Thm}[{\cite[Theorem 1.8]{BMSZ}}] \label{thm:BMSZ}
Let $X$ be a proper variety.
Suppose that $(X,\Delta)$ is log canonical and $-(K_X+\Delta)$ is nef.

If $\gamma (X,\Delta) < 3/2$, then there is a proper finite morphism $Y \to X$ of degree at most two, which is \'{e}tale outside a closed subset of codimension at least two, such that $Y$ is rational.

In particular if $\Cl (X)$ contains no $2$-torsion then $X$ is rational.
\end{Thm}

In \cite{BMSZ}, various examples are provided in order to show that the above criterion is sharp in many aspects (e.g.\ we cannot drop log canonicity of $(X,\Delta)$, nef-ness of $-(K_X+\Delta)$, or the non-existence of $2$-torsion of $\Cl (X)$, etc.).
However, no example is provided to show that the inequality $\gamma < 3/2$ is sharp.
The aim of this section is to show that we cannot relax the inequality $\gamma < 3/2$ in Theorem \ref{thm:BMSZ} at least in dimension $3$.

Let $X = X_6 \subset \mbP (1,1,2,2,3)$ be a very general weighted hypersurfaces of degree $6$ defined over $\mbC$.
We see that the singular locus of $X$ consist of $3$ points $\msp_1,\msp_2,\msp_3$ of type $\frac{1}{2} (1,1,1)$.
Let $H_1, H_2$ be general members of the pencil $|\mcO_X (1)|$ and $D$ a general member of $|\mcO_X (2)|$.
Since $X$ is (very) general, we can assume the following.

\begin{enumerate}
\item $D$ is smooth and it does not pass through $\msp_1,\msp_2,\msp_3$.
\item $H_i$  has a du Val singularity of type $\mathrm{A}_1$ at $\msp_1, \msp_2, \msp_3$ and smooth elsewhere.
\item The scheme-theoretic intersections $H_1 \cap H_2$, $H_1 \cap D$ and $H_2 \cap D$ are nonsingular curves.
\end{enumerate}

We set $\Delta = H_1 + H_2 + \frac{1}{2} D$.

\begin{Lem}
The pair $(X, \Delta)$ is log canonical, $K_X + \Delta \sim_{\mbQ} 0$ and $\gamma (X,\Delta) = 3/2$.
\end{Lem}

\begin{proof}
It is clear that $K_X + \Delta \sim_{\mbQ} 0$.
Since $X$ is $\mbQ$-factorial and is of Picard number $1$, we have $\gamma (X,\Delta) = 3 + 1 - (1 + 1 + 1/2) = 3/2$.

It remains to show that $(X,\Delta)$ is log canonical.
Let $\varphi \colon Y \to X$ be the blowup of $X$ at the points $\msp_1,\msp_2,\msp_3$ and $E_i \cong \mbP^2$ the exceptional divisor over $\msp_i$.
We have
\[
K_Y + \tilde{\Delta} + \frac{1}{2} (E_1 + E_2 + E_3) = \varphi^* (K_X + \Delta),
\]
where $\tilde{\Delta}$ is the proper transform of $\Delta$.
We observe the following:
\begin{itemize}
\item $\tilde{H}_1, \tilde{H}_2, \tilde{D}$ are all smooth.
\item $\tilde{D} \cap E_i = E_j  \cap E_k = \emptyset$ for any $i, j, k$ with $j \ne k$,
\item $\tilde{H}_i$ intersects $E_k$ transversally along a line in $E_k$ for $i = 1,2$ and $k = 1,2,3$, and
\item $\tilde{H}_1 \cap \tilde{H}_2 \cap E_k$ is a point for $k = 1,2,3$.
\end{itemize} 
This means that $\tilde{\Delta} + \frac{1}{2} (E_1+E_2+E_3)$ is a simple normal crossing divisor, and thus $(X,\Delta)$ is log canonical (see also Remark \ref{rem:lcpair} below).
\end{proof}

\begin{Rem} \label{rem:lcpair}
We give an another proof for the log canonicity of $(X,\Delta)$.
Since $X$ and $H_1$, $H_2$, $D$ are general, we see that the support of $\Delta$ is simple normal crossing outside the points $\msp_1, \msp_2, \msp_3$.
Hence $(X,\Delta)$ is log canonical outside $\msp_1, \msp_2, \msp_3$.
The germ $(\msp_i \in X)$ admits an orbifold chart $\tau_i \colon V_i \to (\msp_i \in X)$, where $V_i$ is smooth and $\msp_i \in X$ is the quotient of $V_i$ under a suitable $(\mbZ/2 \mbZ)$-action.
Note that $\tau_i$ is \'etale outside $\msp_i$.
Again by the generality of $X, H_1, H_2, D$, we see that $\tau_i^*\Delta$ has a simple normal crossing support.
This implies that $(V_i, \tau_i^*\Delta)$ is log canonical.
By \cite[8.12 Lemma]{Kollar2}, we conclude that $(X,\Delta)$ is log canonical at $\msp_i$ and the proof is completed.
\end{Rem}

By the main result of this paper, $X$ is not (stably) rational and $\Cl (X) \cong \mbZ$ has no torsion (see Remark \ref{rem:cl}).
Thus the rationality criterion \cite[Theorem 1.8]{BMSZ} in terms of the absolute complexity is sharp, that is, the condition $\gamma (X,\Delta) < 3/2$ cannot be weakened.

\end{document}